\documentclass[final]{siamltex704}
\usepackage{epsfig}
\usepackage{xcolor}
\usepackage{epstopdf}
\usepackage[figuresright]{rotating}
\usepackage[algo2e,ruled,vlined]{algorithm2e}
\usepackage{subfigure}
\usepackage{mathdots,array,booktabs,mathrsfs,url}
\usepackage{multirow}
\usepackage{appendix}
\usepackage{amsfonts}
\usepackage{amsmath}
\usepackage{latexsym}
\usepackage{amssymb}
\usepackage{graphicx,subfigure,color}
\usepackage{amssymb,amsmath,array,blkarray,epsf}
\usepackage{mathdots}
\usepackage{picinpar}
\usepackage{algorithm}
\usepackage{algorithmic}
\usepackage{float}
\usepackage{lipsum}
\usepackage{threeparttable}
\renewcommand{\theenumi}{(\roman{enumi})} 
 
\newtheorem{remark}[theorem]{Remark}
\UseRawInputEncoding

\newcommand{\rmF}{\mathrm{F}}
\newcommand{\rmH}{\mathrm{H}}
\newcommand{\rmT}{\mathrm{T}}
\renewcommand{\tau}{q}
\definecolor{BLUE}{RGB}{0, 0, 255}

\begin{document}
\title{A Novel Adaptive Low-Rank Matrix Approximation Method for Image Compression and Reconstruction\thanks{This work is  supported  in part by the  National Key Research and Development Program of China under grant 2023YFA1010101; the National Natural Science Foundation of China under grants 11971243, 12171210, 12090011, 12271467 and 11771188;  the ``QingLan'' Project for Colleges and Universities of Jiangsu Province (Young and middle-aged academic leaders);
the Major Projects of Universities in Jiangsu Province under grant 21KJA110001;  the Natural Science Foundation of Fujian Province of China under grant 2022J01378;  and the Postgraduate Research \& Practice Innovation Program of Jiangsu Province under grants KYCX24\_1402 and KYCX25\_3153.}
}
\author{
Weiwei Xu\thanks{School of Mathematics and Statistics, Nanjing
University of Information Science and Technology, Nanjing 210044,
P.R. China (e-mail: wwxu@nuist.edu.cn).} \and
Weijie Shen\thanks{School of Mathematics and Statistics, Nanjing
University of Information Science and Technology, Nanjing 210044,
P.R. China (e-mail: swj@nuist.edu.cn).} \and
Chang Liu\thanks{School of Mathematics and Statistics, Jiangsu Normal University, Xuzhou 221116, P.R. China (e-mail: jsnuliuchang@163.com).} 
\and
Zhigang Jia\thanks{Corresponding author. School of Mathematics and Statistics \& RIMS, Jiangsu Normal University, Xuzhou 221116, P.R. China (e-mail: zhgjia@jsnu.edu.cn).}
}

\maketitle
\begin{abstract}
Low-rank matrix approximation plays an important role in various applications such as image processing, signal processing and data analysis. The existing methods require a guess of  the ranks of matrices that represent images  or involve additional  costs to determine the ranks. A novel efficient orthogonal decomposition with automatic basis extraction (EOD-ABE) is proposed to compute the optimal low-rank matrix approximation with adaptive identification of the optimal rank.   By introducing a randomized basis extraction mechanism, EOD-ABE eliminates the need for additional rank determination steps and can compute a rank-revealing approximation to a  low-rank matrix. With a computational complexity of $O(mnr)$, where $m$ and $n$ are the dimensions of the matrix and $r$ is its rank,  EOD-ABE achieves significant speedups compared to the state-of-the-art methods. Experimental results demonstrate the superior speed, accuracy and robustness of EOD-ABE and indicate that EOD-ABE is a powerful tool for fast image compression and reconstruction and hyperspectral image dimensionality reduction in large-scale applications.
\end{abstract}

\section{Introduction}
A well-known image compression and reconstruction method is the low-rank matrix approximation that represents images with matrices and compute their low-rank approximations. A key challenge is the unknown matrix rank, which often requires computationally expensive estimation. This uncertainty hampers the ability to efficiently achieve optimal low-rank approximations. To tackle this issue, we introduce a novel efficient orthogonal decomposition with automatic basis extraction (EOD-ABE) for low-rank matrix approximation, which precisely identifies the optimal rank while delivering high-quality approximations.

Matrix decompositions, including singular value decomposition (SVD), column-pivoted QR (CPQR), and UTV decomposition \cite{hmt11, kc20, kc21, kl18}, are widely used for low-rank approximations to images and provide detailed factorizations of input matrices. Among these, UTV decomposition \cite{kl18} is computationally more efficient than SVD and reveals information about the matrix's numerical null space.  An obstacle to the widespread application of   these methods is that they are computationally intensive, requiring significant resources for eigenvalue and eigenvector computations, pivoting, and iterative processes. Furthermore, applying these methods to large matrices in parallel or distributed environments often demands extensive data transfer and synchronization, limiting their scalability in real-world applications.

Randomized methods \cite{di19, kc19, s19} have gained prominence as efficient alternatives for low-rank approximations due to their computational efficiency, robustness, and high accuracy. The strategy of the randomized method is as follows \cite{hmt11}: first, a random sampling strategy reduces the dimensionality of the input matrix; second, a complete decomposition is performed on the reduced matrix; and finally, the transformed matrix is projected back to the original space. Compared to classical methods, randomized approaches can better leverage parallel architectures, offering a particularly advantage in large-scale scenarios.
A critical limitation of existing randomized methods is their reliance on prior knowledge of the matrix rank. Estimating the rank in advance is computationally intensive and can become a significant bottleneck for large matrices. To overcome this challenge, this paper introduces a novel randomized matrix decomposition algorithm that adaptively extracts the matrix basis and automatically determines its rank.

In recent years, randomized low-rank approximation algorithms have gained attention for their applications in image compression and reconstruction. For example, the projection-based partial QLP (PbP-QLP) algorithm introduced by Kaloorazi and Chen \cite{kc21} employed randomized sampling and non-pivoted QR decomposition to approximate the partial QLP (p-QLP) decomposition, achieving significant computational speedup by removing the need for pivoting. Similarly, Feng and Yu \cite{fy2023} proposed the fast adaptive randomized PCA (farPCA) algorithm, which replaces QR decomposition with matrix multiplications and inversions of small matrices, offering improved performance in multi-threaded environments.  However, the iterative nature and matrix operations in these methods lead to higher computational complexity compared to classical PCA techniques. Unlike these approaches, our method incorporates a rank-revealing mechanism that adaptively determines the optimal rank during image compression and reconstruction. This capability overcomes the limitations of existing methods and delivers superior performance, particularly in scenarios where the matrix rank is not known a priori. The adaptive rank determination capability makes our algorithm a robust and versatile tool for image compression and reconstruction.

Now we illustrate the advantage of our low-rank approximation method with a rank-revealing toy example. In Figure \ref{rgb_sv_chess220} (a),  the color image used for illustration is of size  $220 \times 220$ and  has three color channels denoted by three real matrices $R$, $G$ and $B$. Their ranks  are $50$, $52$ and $53$,  respectively.  Suppose that  such rank information is unknown. Then the existing methods that need pre-decided ranks before computing the low-rank approximations are not suitable to compute the optimal low-rank approximation. Other existing methods that can predict numerical ranks may achieve at the optimal low-rank approximation, but may not find the correct ranks. For instance, Adaptive PCA \cite{dyxl2020} and farPCA \cite{fy2023} predict that the ranks of three color channels are the same $55$. Our method can compute the optimal low-rank approximation and reveal the exact ranks $50$, $52$ and $53$ of three color channels; see Figure \ref{rgb_sv_chess220} (b). Moreover, our method only takes about $0.8$ seconds to obtain such exact low-rank reconstruction. (See Section \ref{sec:numexp} for more numerical comparisons).

Recall that image compression and reconstruction refers to the process of compressing high-dimensional image data via low-rank matrix decomposition, and subsequently reconstructing it from the compressed representation with minimal or without information loss. Our adaptive low-rank approximation method decomposes a matrix $A \in \mathbb{C}^{m \times n}$ into the form $A = UDV^\mathrm{H}$, where $U \in \mathbb{C}^{m \times r}$ and $V \in \mathbb{C}^{n \times r}$ are column-orthogonal matrices, $D \in \mathbb{C}^{r \times r}$ is an upper triangular matrix, and $r$ is the numerical rank of $A$.   If $r<\sqrt{m^2+n^2+4mn+m+n+1/4}-(m+n+1/2)$, then the total number of characters in bytes of three matrix factors  $U$, $D$ and $V$ are  smaller than that of $A$.   In this case, these three matrix factors are suitable for storage and transmission. The reconstruction of the original data only requires multiplying them together. 
For instance,  three color channels of color image in Figure \ref{rgb_sv_chess220} (a) are represented by three nonnegative matrices with $m=n=220$ and their ranks are $r=50,~52$ and $53$, respectively. 
Applying our method, the total number of characters in bytes can be compressed from  $145200$ to about $72284$.
Thus, the total amount of data required to transmit the full-color image is reduced by $50.22$\%.    In fact,  when the image exhibits an approximately low-rank structure, our framework adaptively determines its numerical rank and computes a precise low-rank factorization. This enables lossless compression and accurate reconstruction under the low-rank assumption.

\begin{figure}
  \centering
    \includegraphics[width=5.0in]{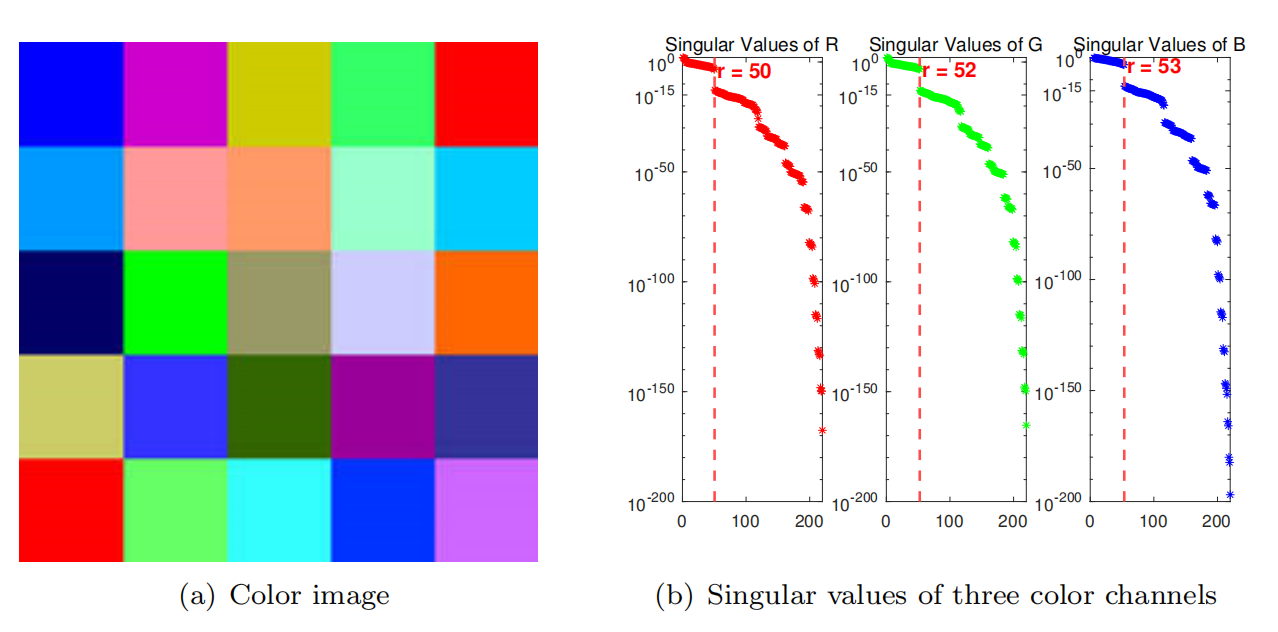}
  \caption{A color image with dimension $220 \times 220$  and the singular values of its color channels (denoted by $R$, $G$ and $B$). The vertical red dotted lines show the numerical ranks calculated by Algorithm \ref{alg:A} with setting a threshold $\varepsilon = 10^{-12}$.}\label{rgb_sv_chess220}
\end{figure}

The contributions of this paper are threefold:
\begin{itemize}
\item
\textbf{Adaptive basis extraction for low-rank approximation with unknown rank}:
This paper proposes a randomized algorithm designed for matrices with unknown rank. The algorithm adaptively constructs orthogonal bases and identifies the numerical rank without relying on prior knowledge. This adaptive mechanism not only reveals the numerical rank but also eliminates the need for computationally expensive pre-estimations of rank, addressing a major limitation in traditional low-rank approximation methods. By determining the optimal rank during the approximation process, the method achieves both high efficiency and broad applicability.
\item
\textbf{Efficient orthogonal decomposition with automatic basis extraction}:
This paper introduces a novel EOD-ABE to compute a low-rank approximation to a matrix with unknown rank. The adaptive basis extraction is naturally combined with the UTV decomposition, allowing the algorithm to improve both computational speed and numerical accuracy while removing the need to predefine the rank.  EOD-ABE provides a robust and effective solution for large-scale low-rank approximation problems in various applications.
\item
\textbf{The computational complexity is reduced to $O(mnr)$}: With a computational complexity of $O(mnr)$, where $m$ and $n$ are the dimensions of the matrix and $r$ is its rank,  EOD-ABE achieves significant speedups compared to the state-of-the-art methods. 
In the hyperspectral image dimensionality reduction, our algorithm reduces computation time by over 20\% compared to existing rank-adaptive methods, while maintaining high accuracy.
\end{itemize}

The rest of this paper is organized as follows. In Section 2, we provide an overview of recent randomized algorithms for low-rank matrix approximation. Section 3 presents low-rank approximation algorithms with unknown rank, accompanied by an analysis of the algorithmic complexity. A detailed error analysis of the proposed algorithm is presented. Section 4 includes several simulation experiments and the application of our algorithm.

\section{Related works}
This section reviews state-of-the-art randomization techniques for low-rank matrix approximation. Over the past few decades, randomized algorithms for low-rank approximations have gained significant attention due to their computational efficiency and suitability for parallel computing, facilitating their application on modern high-performance platforms.

Throughout this paper, by
$\mathbb{R}^{m\times n}$ and $\mathbb{C}^{m\times n}$ we denote the sets of $m\times n$ matrices with entries in the real number field, and $m\times n$ matrices with entries in the complex number field, respectively.
The symbols $I_{n}$ and $O_{m\times n}$ represent the identity matrix of order $n$ and $m\times n$ zero matrix, respectively.
For a real matrix $A\in\mathbb{R}^{m\times n}$, $A^{\rmT}$ denotes its transpose.
For a complex matrix $A\in\mathbb{C}^{m\times n}$, by $A^{\rmH}$, $A^{-1}$, $\mathrm{rank}(A)$ and $\mathrm{tr}(A)$ we denote the conjugate transpose, inverse, rank and trace of matrix $A$, respectively. We use $\|\cdot\|,\|\cdot\|_{2}$ and $\|\cdot\|_{\rmF}$ to denote the unitary invariant norm, spectral norm and Frobenius norm of a matrix, respectively. $A\preceq B$ denotes $A-B$ is negative semi-definite. $A\succeq B$ denotes $A-B$ is positive semi-definite. In this paper, $\mathbb{E}$ specifically denotes expectation with respect to the random test matrix.

\subsection{Randomized SVD}
Halko, Martinsson and Tropp \cite{hmt11} presented a modular framework for constructing randomized algorithms for matrix decompositions, emphasizing the critical role of random sampling in capturing the dominant action of a matrix. Their approach, known as randomized SVD (RSVD), projects the input matrix onto a randomly generated low-dimensional subspace, effectively preserving its key features. The method then applies QR decomposition and SVD to the reduced matrix to construct the low-rank approximation. They also introduced a single-pass variant, the two-sided randomized SVD (TSR-SVD), which processes data in a single iteration. Extensive theoretical analysis and empirical studies confirmed that RSVD and TSR-SVD outperform traditional methods in accuracy, efficiency, and robustness, making them powerful tools for large-scale matrix computations.

An adaptive randomization method is also proposed in \cite{hmt11} and  called adaptive randomized range finder. In this paper, we refer to its combination with RSVD as Adaptive RSVD. The algorithm uses random projection to build an orthogonal basis in an iterative manner while adaptively tuning the sampling strategy to satisfy a prescribed error tolerance. 
During the iterative process, the proposed algorithm increments column by column, which involves calling the underlying code for Schmidt orthogonalization. As a result, the computation time is relatively long in practical computations.
The convergence condition of Adaptive RSVD is achieved by selecting re-orthogonalized vectors and evaluating the relationship between their maximum norm and a threshold, which introduces some computational overhead. We will improve it later in Algorithm \ref{alg:qr}.

Recently,  a fast SVD (FSVD) algorithm that further accelerates low-rank approximation is proposed in  \cite{nsr}.  This method leverages  random renormalization vectors and the Gram-Schmidt process to efficiently approximate the range of the input matrix and employs  a revised Lanczos process to rapidly compute the SVD.

\subsection{Randomized UTV decomposition}
Kaloorazi and Lamare \cite{kl18} introduced the compressed randomized UTV (CoR-UTV) decomposition, which integrates the rank-revealing decomposition and the power method. This decomposition approximates low-rank input matrices through randomized sampling and is defined as follows: 
\begin{definition}[UTV decomposition \cite{s92,s93}] 
For a matrix $A \in \mathbb{C}^{m \times n}$, the UTV decomposition is 
\begin{equation*}
A = UTV^{\rmH},
\end{equation*}
where $U \in \mathbb{C}^{m \times n}$ and $V \in \mathbb{C}^{n \times n}$ have orthonormal columns, and $T$ is a triangular matrix. 
\end{definition}

The CoR-UTV algorithm requires multiple passes over the data and incurs a computational cost of $O(mnr)$ floating-point operations. Designed to harness the capabilities of modern computing architectures, the algorithm exhibits high computational efficiency.

Randomized algorithms for rank-revealing decompositions have seen substantial development in recent years. Martinsson, Quintana-Orti and Heavner \cite{mqh2019} introduced the blocked randomized UTV (randUTV) algorithm, which utilizes a block-structured approach and randomized subspace iteration to efficiently approximate the singular vectors of a matrix. While computationally effective, the practical implementation of randUTV requires advanced memory management and parallel optimization to leverage modern hardware fully.

Kaloorazi and Chen \cite{kc20} proposed the randomized pivoted two-sided orthogonal decomposition (RP-TSOD) algorithm, which applies random projection to reduce the dimensionality of the input matrix, followed by column-pivoted QR (CPQR) decomposition. This method delivers approximate dominant singular bases and singular values, supported by comprehensive error bounds covering the low-rank approximation error, the largest approximate singular values, and the canonical angles between the approximate and exact singular subspaces.

To address the limitations of pivoting in traditional UTV methods, Kaloorazi and Chen \cite{kc21} also developed the PbP-QLP algorithm. This approach eliminates pivoting, enabling it to achieve high computational efficiency while maintaining accuracy. PbP-QLP is particularly well-suited for modern computational architectures, making it a practical choice for large-scale applications. Collectively, these algorithms highlight the versatility and effectiveness of randomized techniques in tackling low-rank approximation problems.

\subsection{Randomized PCA}
Most existing randomized low-rank approximation methods often require the rank $r$ to be specified in advance, a task that is nontrivial in practical applications. To address this limitation, Ding et al. \cite{dyxl2020} introduced a fast adaptive PCA framework capable of automatically determining the dimensionality parameter. This framework significantly accelerates computations for large, sparse matrices, demonstrating its utility in large-scale data analysis. Similarly, Feng and Yu \cite{fy2023} proposed the farPCA algorithm, which optimizes parallel computing efficiency by replacing QR decomposition with matrix multiplication and the inversion of small matrices. While this approach improves computational speed, the choice of block size plays a critical role in ensuring accurate rank approximation.

Given a low-rank matrix $A$ with the rank unknown, there are still no algorithms that can be theoretically proved to yield the  rank-revealing approximation to $A$.

\section{Efficient orthogonal decomposition with Automatic Basis Extraction}
In this section, we construct a novel randomized algorithm EOD-ABE for computing the optimal  low-rank  approximation to $A$. This new algorithm can compute a rank-revealing approximation to a  low-rank matrix $A$.  
The flowchart of the entire algorithm is given in Figure \ref{liuchengtu}.

\begin{figure}
  \centering
      \includegraphics[width=5.0in]{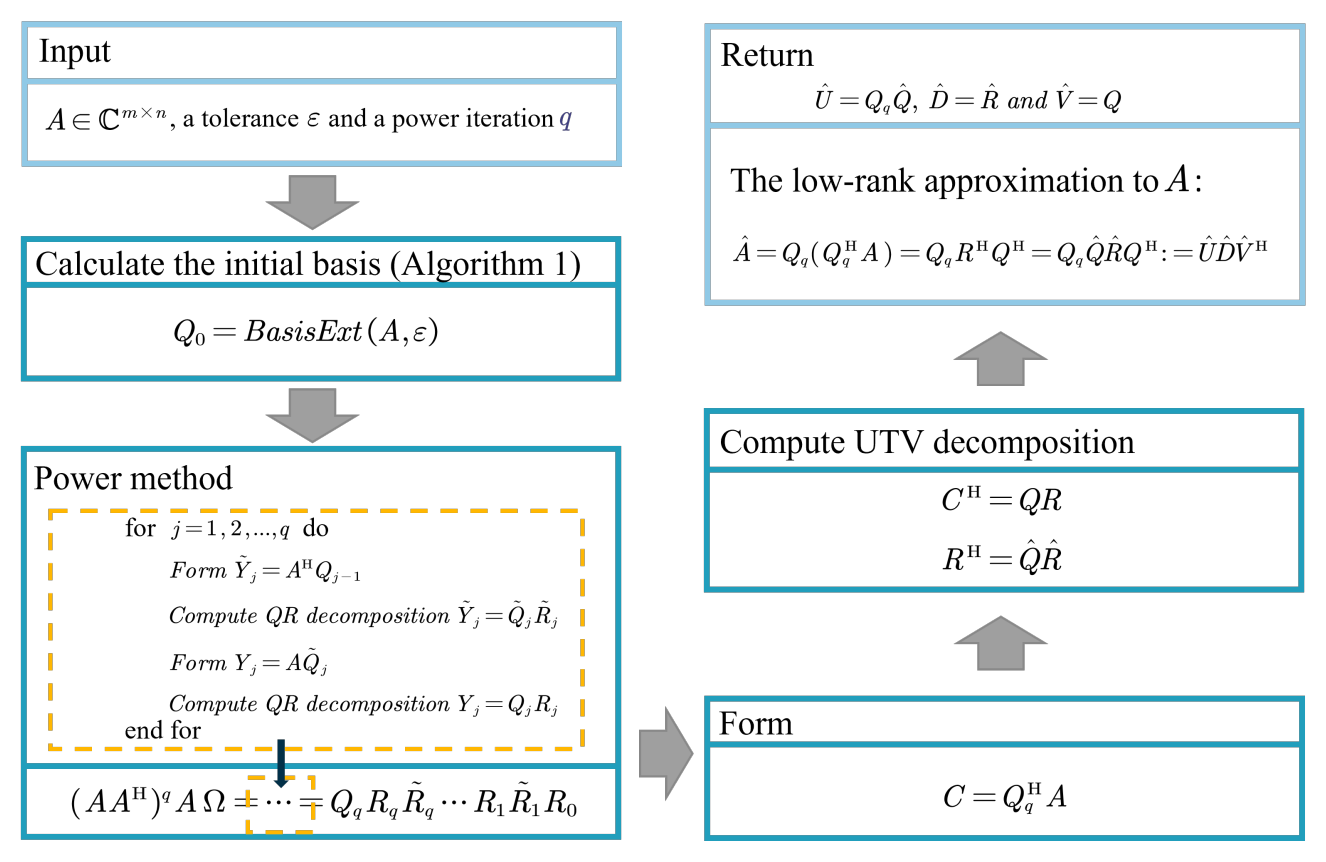}
  \caption{Flowchart of the entire algorithm.}\label{liuchengtu}
\end{figure}

We focus on a matrix $A\in\mathbb{C}^{m\times n}$ with $m\geq n$, noting that the proposed results can be readily generalized to the case $m<n$.  
For convenience, we introduce a rank-revealing UTV decomposition,
\begin{equation}\label{udvd}
A = UDV^{\rmH},
\end{equation}
where $U \in \mathbb{C}^{m \times r}$ and $V \in \mathbb{C}^{n \times r}$ have orthonormal columns, and  $D\in \mathbb{C}^{r \times r}$ is an upper triangular matrix of full rank.
Let $\rank(A)=r$, then a UTV-based optimal rank-revealing approximation to $A$ is defined by 
\begin{equation}\label{udvd2}
\hat{A} = \hat{U}\hat{D}\hat{V}^{\rmH}={\rm arg}\min\limits_{\rank(B)=r}\|B-A\|_{\rmF},
\end{equation}
 where $\hat{D}\in\mathbb{C}^{r\times r}$ is an upper triangular matrix of full rank,  $\hat{U}\in\mathbb{C}^{m\times r}$ and $\hat{V}\in\mathbb{C}^{n\times r}$ have orthonormal columns.

\subsection{Randomized algorithms for basis extraction}
First, we introduce a randomized algorithm for basis extraction that can automatically identify the rank of $A$. 
The main idea is as the  following:
\begin{itemize}
\item  We employ an iterative scheme to extract the basis incrementally in batches since  the size of the approximated basis of $U$ is unknown and can not be  a priori.  Alternatively, the basis could be computed one vector at a time, but this approach is less efficient under modern computer architectures. Therefore, we introduce a block size hyperparameter $k$ to control the batch size, enabling more efficient utilization of the memory hierarchy. As noted in \cite{p21}, selecting $k$ in the range of 10 to 100 is appropriate for many environments.

\item At each iteration of basis extraction,  we multiply $A$ from the left side by a projection matrix generated in the previous iterations, and then sample the columns of the projected matrix by  a standard Gaussian random matrix of size $n \times k$.
 Subsequently, the QR decomposition  the resulted matrix is computed,  yielding an additional batch of basis vectors. The goal is to construct an orthonormal matrix  $Q$ satisfying
$$\|(I - QQ^{\rmH}) A\| \leq\varepsilon.$$
\end{itemize}

The pseudo code is given in Algorithm \ref{alg:qr}, named by {\it BasisExt}.   
This stopping criterion at step 6 is reformulated in terms of evaluating the diagonal elements of the upper triangular matrices $T_j$'s .

\begin{algorithm}
    \caption{Randomized  basis extraction algorithm ({\it BasisExt})}
    \label{alg:qr}
    \renewcommand{\algorithmicrequire}{\textbf{Input:}}
    \renewcommand{\algorithmicensure}{\textbf{Output:}}
    \begin{algorithmic}[1]
    \REQUIRE A low-rank matrix $A\in\mathbb{C}^{m\times n}$, a tolerance $\varepsilon>0$ and a blocksize $1\leq k<n$. Let $s = \lfloor n/k\rfloor$ and $f_{j} = k$, $j=1,2,\ldots,s$, $f_{s+1} = n-sk$. 
    \ENSURE An approximated basis $Q$ of $U$ for $A = UDV^{\rmH}$ as in \eqref{udvd}.   

\STATE $Q = [\;]$.

\FOR {$j=1,2,\ldots,s+1$} \label{for}

\STATE Draw an $n\times f_{j}$ standard Gaussian matrix $\Omega_{j}$.

\STATE Set $Y_{j} = (I_{m}-QQ^{\rmH})A\Omega_{j}$.

\STATE Compute QR decomposition $Y_{j} = P_{j}T_{j}$, where $P_{j} \in \mathbb{C}^{m\times f_j}$, $T_{j} \in \mathbb{C}^{f_j\times f_j}$.

\IF{find the minimum $1\leq \ell\leq f_j$ with $|T_{j}(\ell,\ell)|\leq\varepsilon$} \label{stop}

\STATE $Q=[Q,P_{j}(:,1:\ell-1)]$.

\STATE break.

\ELSE

\STATE $Q=[Q,P_{j}]$.

\ENDIF

\ENDFOR

        \RETURN $Q$. \label{return}
    \end{algorithmic}
\end{algorithm}

\begin{remark}
     In contrast to the algorithms proposed in \cite{dyxl2020,fy2023,hmt11,mqh2019,nsr} that do not require knowing the matrix rank in advance, Algorithm \ref{alg:qr} exhibits several distinct advantages and innovations:
\begin{itemize}
  \item Algorithm \ref{alg:qr} iteratively extracts an orthogonal basis and tracks the diagonal entries of the upper triangular matrix obtained from QR decomposition. Unlike Adaptive PCA \cite{dyxl2020} and farPCA \cite{fy2023}, the numerical rank is not constrained by a predetermined block size. This approach determines the numerical rank automatically and accurately, eliminating the need for additional estimation or manual tuning.
  \item Algorithm \ref{alg:qr} employs a block iteration scheme for basis extraction. This design improves memory efficiency and leverages parallel computing. Unlike the Adaptive RSVD algorithm  \cite{hmt11}, it does not require a preset target rank or oversampling parameters. As a result, it better captures low-rank structures and reduces approximation errors from parameter mismatches.
  \item The randUTV algorithm \cite{mqh2019} requires complex block strategies and memory management, has high implementation difficulty in distributed environments. The theoretical cost of the randUTV of an $m \times n$ matrix is $O(mn^2)$ floating-point operations. Algorithm \ref{alg:qr} sharply reduces communication overhead through localized QR decomposition and small matrix operations.
  \item Theoretically, both FSVD \cite{nsr} and Algorithm \ref{alg:qr} have the computational complexity of $O(mnr)$. But, FSVD adopts a column-wise processing scheme, whereas Algorithm \ref{alg:qr} uses a blocked scheme that enables Level-3 operations, resulting in reduced runtime. 
  In addition,  FSVD requires storing a random matrix with gradually increasing dimensions, while Algorithm~\ref{alg:qr} generates $k$ columns at a time without storage. After obtaining the orthonormal basis, FSVD requires approximately $mr^2 + 2nr^2 + \frac{4}{3}r^3$ floating-point operations to compute the SVD, while Algorithm~\ref{alg:qr} only requires about $mr^2 + 2nr^2 + \frac{2}{3}r^3$ to compute the UTV decomposition. This represents a reduction of roughly $\frac{2}{3}r^3$ in computational cost.
\end{itemize}
   \end{remark}
   
\begin{remark}
  The choice of $\varepsilon$ directly affects the determination of the matrix's approximate rank, as it governs the truncation of small singular values.  For illustrative examples on how varying $\varepsilon$ influences the approximation rank and the overall approximation quality, we refer the reader to Section 4.4.
\end{remark}

Now, we analyze the rank-revealing property of Algorithm \ref{alg:qr}.
\begin{lemma}[\cite{hmt11}]\label{lem1}
Let $\Omega$ be an $m\times n$ standard Gaussian matrix and let $U=[U_{1}^{\rmH}, U_{2} ^{\rmH}]^{\rmH}$ $
\in\mathbb{C}^{m\times m}$ and $V=[V_{1},V_{2}]\in\mathbb{C}^{n\times n}$ be orthonormal matrices. Denote  $U\Omega=[(U_{1}\Omega)^{\rmH}, (U_{2}\Omega )^{\rmH}] ^{\rmH}$ $:=[\Omega_{1}^{\rmH},\Omega_{2}^{\rmH}]^{\rmH}$ and
$\Omega V=\Omega[V_{1},V_{2}]=[\Omega V_{1},\Omega V_{2}]:=[\Lambda_{1},\Lambda_{2}]$. Then we have the following
conclusions.
\begin{enumerate}
  \item If $m \leq n$, then $\Omega$ is a full row rank matrix with probability one. If $m \geq n$,
then $\Omega$ is a full column rank matrix with probability one. Then $\mathrm{rank}(\Omega) =
\min\{m, n\}$.
  \item $U\Omega$, $\Omega V$, $\Omega_{1}$, $\Omega_{2}$, $\Lambda_{1}$ and $\Lambda_{2}$ are also standard Gaussian matrices.
\end{enumerate}
\end{lemma}

\begin{lemma}\label{lem2}
Let $A \in \mathbb{C}^{m\times n}$ $(m\geq n)$ with $\mathrm{rank}(A)=r$ and $\Omega$ be an $n\times d$ standard Gaussian matrix, $1\leq d\leq n$. Then, $\mathrm{rank}(A\Omega)=\min\{r,d\}$.
\end{lemma}

\begin{proof}
Suppose the SVD of $A$ is 
\begin{eqnarray}\label{asvd}
  A = \tilde{U}_{m\times n}\Sigma_{n\times n} \tilde{V}_{n\times n}^{\rmH} 
   = \left[
          \begin{array}{cc}
            \tilde{U}_{r} & \tilde{U}_{0} \\
          \end{array}
        \right]
        \left[
          \begin{array}{cc}
            \Sigma_{r} & O \\
            O & \Sigma_{0} \\
          \end{array}
        \right]
        \left[
          \begin{array}{c}
            \tilde{V}_{r}^{\rmH} \\
            \tilde{V}_{0}^{\rmH} \\
          \end{array}
        \right],
\end{eqnarray}
where $\tilde{U}_{r}\in\mathbb{C}^{m\times r}$, $\tilde{U}_{0}\in\mathbb{C}^{m\times(n-r)}$, $\tilde{V}_{r}\in\mathbb{C}^{n\times r}$ and $\tilde{V}_{0}\in\mathbb{C}^{n\times(n-r)}$ have orthonormal columns, $\Sigma_{r}\in\mathbb{R}^{r\times r}$ and $\Sigma_{0}\in\mathbb{R}^{(n-r)\times(n-r)}$ are diagonal matrices containing the singular values of $A$. 
Since $\mathrm{rank}(A)=r$, we have
\begin{eqnarray*}
A\Omega=\tilde{U}\Sigma \tilde{V}^{\rmH}\Omega
=\left[
                              \begin{array}{cc}
                                \tilde{U}_r & \tilde{U}_0 \\
                              \end{array}
                            \right]
                            \left[
                               \begin{array}{cc}
                                 \Sigma_r & O \\
                                 O  & O\\
                               \end{array}
                             \right]
                             \left[
                               \begin{array}{c}
                                 \tilde{V}_r^{\rmH} \\
                                 \tilde{V}_0^{\rmH} \\
                               \end{array}
                             \right]
                             \Omega
                             =\tilde{U}_r\Sigma_r\tilde{V}_r^{\rmH}\Omega.
\end{eqnarray*}
By Lemma \ref{lem1}, $\tilde{V}_r^{\rmH}\Omega$ is an $r \times d$ standard Gaussian matrix, $\mathrm{rank}(\tilde{V}_r^{\rmH}\Omega)=\min\{r,d\}$. And $\tilde{U}_r\Sigma_r$ is a full column rank matrix. Therefore, $\mathrm{rank}(A\Omega)=\mathrm{rank}(\tilde{U}_r\Sigma_r\tilde{V}_r^{\rmH}\Omega)=\mathrm{rank}(\tilde{V}_r^{\rmH}\Omega)=\min\{r,d\}$.
\end{proof}

\begin{lemma}[\cite{a1996}]
Let $A \in \mathbb{C}^{m \times n}$ $(m \geq n$) with $\mathrm{rank}(A)=r$. Then there exists a permutation matrix $\Pi$, such that
\[A\Pi = Q [R_{11}, R_{12}],\]
where $Q \in \mathbb{C}^{m \times r}$ is a unitary matrix, and $R_{11} \in \mathbb{C}^{r \times r}$ is a non-singular upper triangular matrix.
\end{lemma}

Next, we show that Algorithm \ref{alg:qr} incorporates a random matrix, which eliminates the need for a permutation matrix, and ensures that the diagonal elements of the upper triangular matrix $R$ beyond the matrix rank are zero.

Let $\Omega_j$, $Y_j$ and $T_j$ for $j=1,2,\ldots,z+1$  be defined or computed as in Algorithm \ref{alg:qr}. Denote $\Omega = [\Omega_1, \Omega_2, \cdots, \Omega_{z+1}]$ and $T = \mathrm{diag}\{T_{1},\cdots,T_{z+1}\}$. 

\begin{theorem}\label{th3.4}
Let $A \in \mathbb{C}^{m\times n}$ $(m \ge n)$ with $\rank(A)=r$ and let $\Omega = [\Omega_1, \Omega_2, \cdots, \Omega_{z+1}]$ be an $n\times d~(d\geq r)$ standard Gaussian matrix.
Suppose that $A\Omega$ has the QR decomposition 
\begin{eqnarray}\label{aoqr}
A\Omega = A[\Omega_{1},\cdots,\Omega_{z+1}] 
  = [Q_{1},\cdots,Q_{z+1}]\left[
                                                   \begin{array}{ccc}
                                                     R_{11} & \cdots & R_{1,z+1} \\
                                                      & \ddots & \vdots \\
                                                      &  & R_{z+1,z+1} \\
                                                   \end{array}
                                                 \right]. 
\end{eqnarray}
Then  the following conclusions hold with probability one.
\begin{enumerate}
  \item   The QR decomposition  $Y_j=P_{j}T_{j}$ satisfies $P_{j}=Q_j$ and $T_j=R_{jj}$, $j=1,2,\ldots,z+1$.
  \item $|T(i,i)|>0, i=1,2,\ldots,r$, and $|T(i,i)|=0, i=r+1,r+2,\ldots,d$.
\end{enumerate}
\end{theorem}
\begin{proof}
(i) From Lemma \ref{lem2}, we know that $A\Omega_1$, $A\Omega_2$, $\ldots$, $A\Omega_{z+1}$ are of full column rank with probability one.    
For $j=1$, $Y_1=A\Omega_{1} = P_{1}T_{1}$ satisfies $P_{1}=Q_1$, $T_1=R_{11}$.
For $j=2$, we can get $Y_2=A\Omega_{2}-Q_{1}Q_{1}^{\rmH}A\Omega_{2} = P_{2}T_{2}$ according to Algorithm \ref{alg:qr}.
By \eqref{aoqr}, it follows that
$A\Omega_{2}=Q_{1}R_{12}+Q_{2}R_{22}$.
Then we have
\[Q_{1}Q_{1}^{\rmH}A\Omega_{2} = Q_{1}Q_{1}^{\rmH}Q_{1}R_{12} + Q_{1}Q_{1}^{\rmH}Q_{2}R_{22} = Q_{1}R_{12}.\]
That means $Y_{2} = P_{2}T_{2}$ satisfies $P_{2}=Q_{2}$ and $T_{2}=R_{22}$.\\
\indent Assume that the assertion $Y_{j} = P_{j}T_{j}$ satisfies $P_{j}=Q_{j}$, $T_{j}=R_{jj}$ holds for $2\leq j\leq z$. 
Then for $j+1$, we can get
\begin{equation*}
Y_{j+1} = \left(I - \sum_{i=1}^{j}Q_{i}Q_{i}^{\rmH}\right)A\Omega_{j+1} = P_{j+1}T_{j+1}
\end{equation*}
according to Algorithm \ref{alg:qr}.
By \eqref{aoqr}, it follows that
\begin{equation*}
\begin{aligned}
A\Omega_{j+1} = Q_{1}R_{1,j+1}+Q_{2}R_{2,j+1}+\cdots+Q_{j+1}R_{j+1,j+1}.
\end{aligned}
\end{equation*}
Since $Q_{1},Q_{2},\ldots,Q_{j}$ are orthonormal, we know that the projection operation $(I - \sum_{i=1}^{j}Q_{i}Q_{i}^{\rmH})$ will only retain the part of $A\Omega_{j+1}$ that corresponds to $Q_{j+1}$. So we obtain
$Y_{j+1} = Q_{j+1}R_{j+1,j+1}$.
Therefore, $Y_{j+1}=P_{j+1}T_{j+1}$ satisfies $P_{j+1} = Q_{j+1}$ and $T_{j+1}=R_{j+1,j+1}$. We have proved conclusion (i) by mathematical induction. 

(ii) When Algorithm \ref{alg:qr} does not terminate, we can see that as $j$ increases, Lemma \ref{lem2} shows that $A[\Omega_{1},\cdots,\Omega_{j}]$ is always full column rank with probability one, and
\begin{eqnarray*}
\mathrm{rank}(\mathrm{diag}\{T_{1},\cdots,T_{j}\}) = \mathrm{rank}(\mathrm{diag}\{R_{11},\cdots,R_{jj}\}) = \mathrm{rank}\left(A[\Omega_{1},\cdots,\Omega_{j}]\right)
\end{eqnarray*}
holds true.
Since Algorithm \ref{alg:qr} breaks only if there exists $\ell$ at the $(z+1)$th step such that $|T_{z+1}(\ell,\ell)|=0$, there must be $zk+\ell-1=r$.
This implies $|T(i,i)|>0$, $i = 1,2,\ldots,r$, and $|T(i,i)|=0, i=r+1,r+2,\ldots,d$.
\end{proof}

\begin{remark}\label{rem3.6}
  In practical applications, we determine the numerical rank $r$ of a matrix based on the relationship between the diagonal elements of $T$ and a given tolerance $\varepsilon>0$. That is, $|T(r,r)|>\varepsilon$ and $|T(r+1,r+1)|\leq\varepsilon$.
\end{remark}

\begin{remark}\label{rem3.7}
Based on Theorem \ref{th3.4} and Remark \ref{rem3.6}, permutation matrices are not required to ensure that the diagonal elements beyond the matrix rank in the upper triangular matrix are zero.
Let $A \in \mathbb{C}^{m \times n}$ $(m \geq n)$ with $\mathrm{rank}(A)=r$ and let $\Omega$ be an $n\times n$ standard Gaussian matrix. Then
\[A\Omega = Q [R, \tilde{R}],\]
where $Q \in \mathbb{C}^{m \times r}$ is of full column rank, and $R \in \mathbb{C}^{r \times r}$ is a non-singular upper triangular matrix.
\end{remark}

\subsection{Efficient orthogonal decomposition with automatic basis extraction}
In this part, we present an EOD-ABE to compute a low-rank approximation to a matrix with unknown rank. The main steps are outlined in Algorithm \ref{alg:A}.

\begin{algorithm}[H]
    \caption{EOD-ABE via subspace iteration}
    \label{alg:A}
    \renewcommand{\algorithmicrequire}{\textbf{Input:}}
    \renewcommand{\algorithmicensure}{\textbf{Output:}}
    \begin{algorithmic}[1]
        \REQUIRE $A\in\mathbb{C}^{m\times n}$, a proper tolerance $\varepsilon$ and a power iteration $\tau$.  
        \ENSURE Low-rank approximation: $\hat{A}=\hat{U}\hat{D}\hat{V}^{\rmH}$, where $\hat{D}\in\mathbb{C}^{r\times r}$ is an upper triangular matrix, $\hat{U}\in\mathbb{C}^{m\times r}$ and $\hat{V}\in\mathbb{C}^{n\times r}$ have orthonormal columns.  
        \STATE Compute an approximated basis $Q_{0} = BasisExt(A, \varepsilon)$ by Algorithm \ref{alg:qr}.
        \FOR {$j=1,2,\ldots,\tau$}
            \STATE Form $\tilde{Y}_{j}=A^{\rmH}Q_{j-1}$.
            \STATE Compute the QR decomposition $\tilde{Y}_{j}=\tilde{Q}_{j}\tilde{R}_{j}$.
            \STATE Form $Y_{j}=A\tilde{Q}_{j}$.
            \STATE Compute the QR decomposition $Y_{j}=Q_{j}R_{j}$.
        \ENDFOR
        \STATE Form $C=Q_{\tau}^{\rmH}A$.\label{CQA}
        \STATE Compute the QR decomposition $C^{\rmH}=QR$.\label{C1QR}
        \STATE Compute the  QR decomposition $R^{\rmH}=\hat{Q}\hat{R}$.\label{RQR}

        \RETURN $\hat{U}=Q_{\tau}\hat{Q}$, $\hat{D}=\hat{R}$ and $\hat{V}=Q$.\label{AUDV}
    \end{algorithmic}
\end{algorithm}

  A major distinction of Algorithm~\ref{alg:A}, compared with RSVD~\cite{hmt11}, TSR-SVD~\cite{hmt11}, CoR-UTV~\cite{kl18}, RP-TSOD~\cite{kc20}, and PbP-QLP~\cite{kc21}, lies in how the basis for $\mathrm{range}(A)$ is computed.
  The existing methods require a user-specified sampling parameter (e.g., target rank or oversampling size), which is often not known in advance and may require additional cost to estimate. In contrast, our algorithm adaptively constructs the basis based on the actual structure of the input matrix, without any prior knowledge of the matrix rank. This adaptivity enhances both robustness and efficiency in practical applications.
  
When the input matrix $A$ exhibits flat singular values-that is, its singular values decay slowly-or when $A$ is very large, the relative contribution of the trailing singular components to the dominant singular directions can be suppressed by exponentiating the matrix. To achieve this, Algorithm~\ref{alg:A} employs subspace iteration. First, a small integer $q > 0$ is fixed, which controls how many steps of power iteration will be taken. We apply the randomized sampling scheme to the matrix $Y = (A A^{\rmH})^q A$. The matrix $Y$ shares the same singular vectors as the input matrix $A$, but its singular values decay much more rapidly \cite{hmt11}: $$\sigma_i(Y) = \sigma_i(A)^{2q+1},$$ for $i = 1, 2, 3, \ldots$.
  The choice of the power iteration parameter $q$ governs the balance between accuracy and efficiency. In most cases, $q = 0, 1$, or $2$ is sufficient. When $q = 0$, the algorithm skips power iterations altogether and achieves the highest computational efficiency, though at the cost of reduced accuracy. This setting is suitable for scenarios where speed is prioritized and approximate results are acceptable. When $q = 1$, one additional multiplication by $A A^{\rmH}$ is introduced. The runtime increases slightly, but the accuracy improves significantly, making it a good compromise for applications that demand both speed and higher accuracy. For larger values of $q$, the approximation continues to be refined, but the computational cost grows, while the marginal gain in accuracy becomes negligible. This trade-off is theoretically justified by our error bound in Theorem~\ref{siam} and empirically supported by the experiments in Section~4.5, which recommend using $q = 0$ or 1 in most practical settings. Since the approximation rank of $A$ is determined in the first step of Algorithm~\ref{alg:A} (equal to the number of columns of $Q_0$), the subsequent steps~2-7 only require QR decompositions on the iterated subspace. There is therefore no need to invoke Algorithm~\ref{alg:qr}, further enhancing the overall efficiency of our approach.

Now, we show that the low-rank approximation computed by Algorithm \ref{alg:A} reveals the rank of $A$. 
\begin{theorem}
Let $A\in\mathbb{C}^{m\times n}$ $(m\geq n)$ with $\rank(A)=r$. Then Algorithm \ref{alg:A} with a proper tolerance $\varepsilon>0$ computes a rank-revealing approximation to $A$:
\[\hat{A} = \hat{U}\hat{D}\hat{V}^{\rmH},\]
where $\hat{U}=Q_{\tau}\hat{Q}\in\mathbb{C}^{m\times r}$ and $\hat{V}=Q\in\mathbb{C}^{n\times r}$ are column orthogonal, and $\hat{D}=\hat{R}\in\mathbb{C}^{r\times r}$ is an upper triangular matrix of rank $r$.

\end{theorem}
\begin{proof}
From Theorem \ref{th3.4} and Remark \ref{rem3.6}, we know that
one can get an approximated basis $Q_{0}\in\mathbb{C}^{n\times r}$ of the left singular vectors of $A$ from Algorithm \ref{alg:qr}. Then, 
\begin{equation*}
\begin{aligned}
(AA^{\rmH})^{\tau}A\Omega &= (AA^{\rmH})^{\tau}Q_{0}R_{0} 
= (AA^{\rmH})^{\tau-1}AA^{\rmH}Q_{0}R_{0} 
= (AA^{\rmH})^{\tau-1}A\tilde{Q}_{1}\tilde{R}_{1}R_{0} \\
&= (AA^{\rmH})^{\tau-1}Q_{1}R_{1}\tilde{R}_{1}R_{0} 
= (AA^{\rmH})^{\tau-2}AA^{\rmH}Q_{1}R_{1}\tilde{R}_{1}R_{0} \\
&= (AA^{\rmH})^{\tau-2}A\tilde{Q}_{2}\tilde{R}_{2}R_{1}\tilde{R}_{1}R_{0} \\
&= (AA^{\rmH})^{\tau-2}Q_{2}R_{2}\tilde{R}_{2}R_{1}\tilde{R}_{1}R_{0} \\
&= \cdots \\
&= Q_{\tau}R_{\tau}\tilde{R}_{\tau}\cdots R_{1}\tilde{R}_{1}R_{0}.
\end{aligned}
\end{equation*}
Construct a matrix
\begin{equation*}
C=Q_{\tau}^{\rmH}A,
\end{equation*}
which can be seen as linear combinations of $A$'s rows by means of $Q_{q}$. Let the QR decomposition of $C^{\rmH}$ be
\begin{equation}\label{CQR}
C^{\rmH}=QR,
\end{equation}
where $Q \in \mathbb{C}^{n\times r}$ is column-orthonormal, $R\in \mathbb{C}^{r\times r}$ is a full rank upper triangular matrix. We call the diagonals of $R$ by R-values. And let the QR decomposition of $R^{\rmH}$ be
\begin{equation*}
R^{\rmH}=\hat{Q}\hat{R},
\end{equation*}
where $\hat{Q} \in \mathbb{C}^{r\times r}$ is an orthonormal matrix, $\hat{R}\in \mathbb{C}^{r\times r}$ is a full rank upper triangular matrix.
Then, we obtain a rank-revealing approximation to $A$:
\begin{eqnarray*}
\hat{A}=Q_{\tau}(Q_{\tau}^{\rmH}A)
=Q_{\tau}R^{\rmH}Q^{\rmH} 
=Q_{\tau}\hat{Q}\hat{R}Q^{\rmH} := \hat{U}\hat{D}\hat{V}^{\rmH},
\end{eqnarray*}
where $\hat{U}=Q_{\tau}\hat{Q}$, $\hat{D}=\hat{R}$ and $\hat{V}=Q$.
\end{proof}

\subsection{Error analysis of the low-rank appoximation}
In this part, we  develop bounds for the errors of the low-rank approximation.  

Several lemmas are recalled from \cite{hmt11}. 
Recall that an orthogonal projector is a nonzero Hermitian matrix $P$ that satisfies the polynomial $P^{2}=P$.
\begin{lemma}[\cite{hmt11}]\label{lem:pm}
For a given matrix $A$,  let $P_{A}$ denote the unique orthogonal
projector with $\mathrm{range}(P_{A})=\mathrm{range}(A)$. When $A$ has
full column rank,  this projector is explicitly expressed as
\begin{eqnarray*}
P_{A} = A(A^{\rmH} A)^{-1}A^{\rmH}.
\end{eqnarray*}
For a unitary matrix $Q$, there is
\begin{eqnarray*}
Q^{\rmH}P_{A}Q = P_{Q^{\rmH}A}.
\end{eqnarray*}
\end{lemma}

\begin{lemma}[\cite{hmt11}]\label{lem:p}
Suppose $\mathrm{range}(N) \subset \mathrm{range}(M)$. Then, for each matrix $A$, it holds that $\|P_{N}A\| \leq \|P_{M}A\|$ and that $\|(I - P_{M})A\| \leq \|(I - P_{N})A\|$.
\end{lemma}

\begin{lemma}[\cite{hmt11}]\label{lem:sgt}
Fix matrices $S$, $T$, and draw a standard Gaussian matrix $\Omega$. Then
$\left(\mathbb{E}\|S\Omega T\|_{\rmF}^{2}\right)^{1/2} = \|S\|_{\rmF}\|T\|_{\rmF}$.
\end{lemma}

\begin{lemma}[\cite{hmt11}]\label{lem:ginv}
Draw an $m \times n$ standard Gaussian matrix $\Omega$ with $n-m \geq 2$. Then
\begin{equation*}
  \mathbb{E}\|\Omega^{\dagger}\|_{\rmF}^{2} = \frac{m}{n-m-1}.
\end{equation*}
\end{lemma}

\begin{lemma}[\cite{hmt11}]\label{lem:8.2}
Suppose that $M\succeq O$. Then 
$I-(I+M)^{-1}\preceq M$.
\end{lemma}

Now we present the error bounds  of the low-rank approximation.
\begin{theorem}\label{siam}
Let $A$ be an $m\times n$ $(m\geq n)$ matrix with singular values $\sigma_{1}\geq\sigma_{2}\geq\cdots\geq\sigma_{n}$, $\mathrm{rank}(A)=r$ and have the SVD of the form \eqref{asvd}, and let $\Omega$ be an $n\times d$ standard Gaussian matrix.
Let $\hat{A}$ be the low-rank approximation to $A$ computed through Algorithm \ref{alg:A}. If $d-r\geq2$, then 
\begin{equation*}
  \|A-\hat{A}\|_{\rmF}^{2} \leq \alpha^{4\tau}\|\Sigma_{0}\Lambda_{2}\Lambda_{1}^{\dagger}\|_{\rmF}^{2}+\|\Sigma_{0}\|_{\rmF}^{2}
\end{equation*}
and
\begin{equation}\label{bound}
    \mathbb{E}\|A-\hat{A}\|_{\rmF} \leq \left(1+\frac{r\alpha^{4\tau}}{d-r-1}\right)^{\frac{1}{2}}\left(\sum_{j>r}\sigma_{j}^{2}\right)^{\frac{1}{2}},
\end{equation}
where $\alpha=\frac{\sigma_{r+1}}{\sigma_{r}}$, $\Lambda_{1}=V_{r}^{\rmH}\Omega$ and $\Lambda_{2}=V_{0}^{\rmH}\Omega$.
\end{theorem}

\begin{proof}
By Lemma \ref{lem:pm}, we write
\begin{equation}\label{th3.3:1}
  \|A-\hat{A}\|_{\rmF}=\|A-Q_{\tau}Q_{\tau}^{\rmH}A\|_{\rmF}=\|(I-P_{Q_{\tau}})A\|_{\rmF}.
\end{equation}
We observe that $B$ is represented as
\begin{equation*}
  B = (AA^{\rmH})^{\tau}A\Omega = \tilde{U}\left[
                   \begin{array}{cc}
                     \Sigma_{r}^{2\tau+1} & O \\
                     O & \Sigma_{0}^{2\tau+1} \\
                   \end{array}
                 \right]\tilde{V}^{\rmH}\Omega.
\end{equation*}
Let $\tilde{V}^{\rmH}\Omega = \left[
                     \begin{array}{c}
                       \tilde{V}_{r}^{\rmH}\Omega \\
                       \tilde{V}_{0}^{\rmH}\Omega \\
                     \end{array}
                   \right] :=
                   \left[
                     \begin{array}{c}
                       \Lambda_{1} \\
                       \Lambda_{2} \\
                     \end{array}
                   \right]
$.
Now, we form $\hat{B}$ as
\begin{equation*}
  \hat{B} = \tilde{U}^{\rmH}B = \left[
                       \begin{array}{c}
                         \Sigma_{r}^{2\tau+1}\Lambda_{1} \\
                         \Sigma_{0}^{2\tau+1}\Lambda_{2} \\
                       \end{array}
                     \right].
\end{equation*}
We therefore have
\begin{equation*}
  \mathrm{range}(\hat{B})=\mathrm{range}(\tilde{U}^{\rmH}B)=\mathrm{range}(\tilde{U}^{\rmH}Q_{\tau}).
\end{equation*}
Form another matrix $\tilde{B}$ by shrinking the subspace of $\hat{B}$:
\begin{equation}\label{is}
  \tilde{B} = \hat{B}\Lambda_{1}^{\dagger}\Sigma_{r}^{-(2\tau+1)} = \left[
                           \begin{array}{c}
                             I \\
                             S \\
                           \end{array}
                         \right],
\end{equation}
where $S = \Sigma_{0}^{2\tau+1}\Lambda_{2}\Lambda_{1}^{\dagger}\Sigma_{r}^{-(2\tau+1)}$.
The construction \eqref{is} ensures that
\begin{equation}\label{range}
  \mathrm{range}(\tilde{B}) \subset \mathrm{range}(\hat{B}) = \mathrm{range}(\tilde{U}^{\rmH}Q_{\tau}).
\end{equation}
By Lemmas \ref{lem:p} and \ref{lem:8.2}, we get
\begin{equation}\label{po4}
\begin{aligned}
  \|I-P_{\tilde{U}^{\rmH}Q_{\tau}}\|_{\rmF}\leq\|I-P_{\tilde{B}}\|_{\rmF} &= \left\| \left[\begin{array}{cc}
I-(I+S^{\rmH}S)^{-1}&-(I+S^{\rmH}S)^{-1}S^{\rmH}\\
-S(I+S^{\rmH}S)^{-1}&I-S(I+S^{\rmH}S)^{-1}S^{\rmH}
\end{array}\right] \right\|_{\rmF} \\
& \leq \left\|\left[\begin{array}{cc}
S^{\rmH}S&-(I+S^{\rmH}S)^{-1}S^{\rmH}\\-S(I+S^{\rmH}S)^{-1}&I
\end{array}\right]\right\|_{\rmF}.
\end{aligned}
\end{equation}
For the Frobenius norm, Lemmas \ref{lem:pm} and \ref{lem:p}, \eqref{asvd}, \eqref{range} and \eqref{po4}, we have
\begin{equation*}
\begin{aligned}
\|(I-P_{Q_{\tau}})A\|_{\rmF}^{2} &= \|\tilde{U}^{\rmH}(I-P_{Q_{\tau}})\tilde{U}\Sigma \tilde{V}^{\rmH}\|_{\rmF}^{2}
= \|(I-P_{\tilde{U}^{\rmH}Q_{\tau}})\Sigma \tilde{V}^{\rmH}\|_{\rmF}^{2}\\
&\leq\|(I-P_{\tilde{B}})\Sigma\|_{\rmF}^{2}
=\mathrm{tr}(\Sigma^{\rmH}(I-P_{\tilde{B}})\Sigma)\\
&\leq \mathrm{tr}(\Sigma_{r}^{\rmH}S^{\rmH}S\Sigma_{r})+\mathrm{tr}(\Sigma_{0}^{\rmH}\Sigma_{0})
=\|S\Sigma_{r}\|_{\rmF}^{2}+\|\Sigma_{0}\|_{\rmF}^{2}\\
&\leq\alpha^{4\tau}\|\Sigma_{0}\Lambda_{2}\Lambda_{1}^{\dagger}\|_{\rmF}^{2}+\|\Sigma_{0}\|_{\rmF}^{2},
\end{aligned}
\end{equation*}
where $\alpha=\frac{\sigma_{r+1}}{\sigma_{r}}$. Recall that $\Lambda_{1} = V_{r}^{\rmH}\Omega$ and $\Lambda_{2} = V_{0}^{\rmH}\Omega$. The Gaussian distribution is rotationally invariant, so $V^{\rmH}\Omega$ is also a standard Gaussian
matrix. Observe that $\Lambda_{1}$ and $\Lambda_{2}$ are nonoverlapping submatrices of $V^{\rmH}\Omega$. So
these two matrices are not only standard Gaussian but also stochastically independent.
By \eqref{th3.3:1} and Lemmas \ref{lem:sgt} and \ref{lem:ginv}, we have
\begin{equation*}
\begin{aligned}
    \mathbb{E}\|A-\hat{A}\|_{\rmF} &\leq \sqrt{\alpha^{4\tau}\mathbb{E}\|\Sigma_{0}\Lambda_{2}\Lambda_{1}^{\dagger}\|_{\rmF}^{2}+\|\Sigma_{0}\|_{\rmF}^{2}}
    = \sqrt{\alpha^{4\tau}\|\Sigma_{0}\|_{\rmF}^{2}\mathbb{E}\|\Lambda_{1}^{\dagger}\|_{\rmF}^{2} + \|\Sigma_{0}\|_{\rmF}^{2}}\\
    &= \left(1+\frac{r\alpha^{4\tau}}{d-r-1}\right)^{\frac{1}{2}}\|\Sigma_{0}\|_{\rmF}
    = \left(1+\frac{r\alpha^{4\tau}}{d-r-1}\right)^{\frac{1}{2}}\left(\sum_{j>r}\sigma_{j}^{2}\right)^{\frac{1}{2}}.
\end{aligned}
\end{equation*}
\end{proof}

\begin{remark}
When $q=0$, the error bound on the right side of \eqref{bound} reduced to the error bound in Theorem 10.5 of \cite{hmt11}, i.e.,
  $\mathbb{E}\|A-\hat{A}\|_{\mathrm{F}}\leq\left(1+\frac{r}{d-r-1}\right)^{\frac{1}{2}}\left(\sum_{j>r}\sigma_j^2\right)^{\frac{1}{2}}$.
  When $q\geq 1$, the error bound \eqref{bound} of this paper is sharper than that in \cite{hmt11}.

In \cite{kc20}, the error bound proposed by Kaloorazi and Chen is 
  \[\mathbb{E}\|A-\hat{A}\|_{\rmF} \leq \left(1+\left(\frac{r\alpha^{4\tau}}{d-r-1}\right)^{\frac{1}{2}}\right)\left(\sum_{j>r}\sigma_{j}^{2}\right)^{\frac{1}{2}}.\]
Clearly, our error bound \eqref{bound} is also sharper than it.
\end{remark}

The following theorem presents the perturbation bound between the diagonal elements of the matrix $\hat{D}$ and the singular values of the matrix $A$.

\begin{theorem}
Let $A \in \mathbb{C}^{m \times n}$ have the SVD of the form \eqref{asvd}.
Let the orthogonal decomposition of $\hat{A} = Q_{\tau}Q_{\tau}^{\rmH}A$ be $\hat{A} = \hat{U}\hat{D}\hat{V}^{\rmH}$,
where $\hat{U}\in \mathbb{C}^{m \times r}$ and $\hat{V} \in \mathbb{C}^{n \times r}$ are orthonormal matrices, and $\hat{D} \in \mathbb{C}^{r \times r}$ is an upper triangular matrix. Then the following bound holds:
\begin{equation*}
\begin{aligned}
\max_{1\leq i\leq r}|d_{ii} - \sigma_{i}| \leq \|A\|_{\rmF} (\|Q^{\bot}\|_{2} + \|\hat{U}^{\rmH}\tilde{U} - I\|_{2} + \|\hat{V}^{\rmH} \tilde{V} - I\|_{2}),
\end{aligned}
\end{equation*}
where $d_{ii}$ is the $i$th diagonal element of $\hat{D}$, $\sigma_{i}$ is the $i$th singular value of $A$, $Q^{\bot}=I-Q_{\tau}Q_{\tau}^{\rmH}$ gradually approaches zero matrix as the number of columns of $Q_{\tau}$ increases to $n$.
\end{theorem}

\begin{proof}
Substituting the SVD of $A$ into $\hat{A} = Q_{\tau}Q_{\tau}^{\rmH}A$ and $\hat{A} = \hat{U}\hat{D}\hat{V}^{\rmH}$, we have
\[
\hat{D} = \hat{U}^{\rmH} Q_{\tau}Q_{\tau}^{\rmH} (\tilde{U} \Sigma \tilde{V}^{\rmH}) \hat{V}.
\]
Then
\[
\hat{D} - \Sigma = \hat{U}^{\rmH} Q_{\tau}Q_{\tau}^{\rmH} \tilde{U} \Sigma \tilde{V}^{\rmH} \hat{V} - \Sigma.
\]
Adding and subtracting $I$, we decompose the difference as
\[
\hat{D} - \Sigma = (\hat{U}^{\rmH} Q_{\tau}Q_{\tau}^{\rmH} \tilde{U} - I) \Sigma \tilde{V}^{\rmH} \hat{V} + \Sigma (\tilde{V}^{\rmH} \hat{V} - I).
\]
Since
\begin{equation*}
\begin{aligned}
\|\hat{U}^{\rmH} Q_{\tau}Q_{\tau}^{\rmH} \tilde{U} - I\|_{2} &= \|\hat{U}^{\rmH} (Q_{\tau}Q_{\tau}^{\rmH} - I) \tilde{U} + \hat{U}^{\rmH}\tilde{U} - I\|_{2} \\
&\leq \|\hat{U}^{\rmH} (Q_{\tau}Q_{\tau}^{\rmH} - I) \tilde{U}\|_{2} + \|\hat{U}^{\rmH}\tilde{U} - I\|_{2} \\
&\leq \|Q^{\bot}\|_{2} + \|\hat{U}^{\rmH}\tilde{U} - I\|_{2},
\end{aligned}
\end{equation*}
we can get
\begin{equation*}
\begin{aligned}
\max_{1\leq i\leq r}|d_{ii} - \sigma_{i}|&\leq \|\hat{D} - \Sigma\|_{\rmF}
\leq \|(\hat{U}^{\rmH} Q_{\tau}Q_{\tau}^{\rmH} \tilde{U} - I) \Sigma \tilde{V}^{\rmH} \hat{V}\|_{\rmF} + \|\Sigma (\tilde{V}^{\rmH} \hat{V} - I)\|_{\rmF} \\
&\leq \|\hat{U}^{\rmH} Q_{\tau}Q_{\tau}^{\rmH} \tilde{U} - I\|_{2} \|\Sigma\|_{\rmF} \|\tilde{V}^{\rmH} \hat{V}\|_{2} 
 + \|\Sigma\|_{\rmF} \|\tilde{V}^{\rmH} \hat{V} - I\|_{2} \\
&\leq (\|Q^{\bot}\|_{2} + \|\hat{U}^{\rmH}\tilde{U} - I\|_{2}) \|\Sigma\|_{\rmF} 
 + \|\Sigma\|_{\rmF} \|\tilde{V}^{\rmH} \hat{V} - I\|_{2} \\
&\leq \|A\|_{\rmF} (\|Q^{\bot}\|_{2} + \|\hat{U}^{\rmH}\tilde{U} - I\|_{2} + \|\hat{V}^{\rmH} \tilde{V} - I\|_{2}),
\end{aligned}
\end{equation*}
where $Q^{\bot}=I-Q_{\tau}Q_{\tau}^{\rmH}$ gradually approaches zero matrix as the number of columns of $Q_{\tau}$ increases to $n$.
\end{proof}

\subsection{Computational cost} Now, we analyze the arithmetic and communication costs of our algorithm.

\textbf{Arithmetic Cost.} To compute an approximation of matrix $A$, Algorithm \ref{alg:A} requires the following arithmetic operations:

\begin{table}[H]
\centering
\begin{tabular}{ll}
  \hline
  Step & Computational complexity \\
  \hline
  Computing $Q_{0}$ by Algorithm \ref{alg:qr} & $mnr+4mrk$ \\
  Forming $C$ in step \ref{CQA} & $mnr$ \\
  Computing $Q$ and $R$ in step \ref{C1QR} & $2nr^2$\\
  Computing $\hat{Q}$ and $\hat{R}$ in step \ref{RQR} & $2r^3$\\
  Computing $\hat{U}$, $\hat{D}$ and $\hat{V}$ in step \ref{AUDV} & $mr^2$\\
  \hline
\end{tabular}
\end{table}

The total computational complexity of the above steps is $2mnr+mr^2+2nr^2+2r^3+4mrk$.

After subspace iteration, the complexity of step 2-7 of Algorithm \ref{alg:A} increases by
$$\tau(2mnr+2mr^2+2nr^2).$$
Therefore, the total complexity of Algorithm \ref{alg:A} is 
$$(\tau+1)(2mnr+mr^2+2nr^2)+\tau mr^{2}+2r^3+4mrk.$$

\textbf{Communication Cost.} Communication costs arise from data movement between processors operating in parallel and across different levels of the memory hierarchy. On modern computing platforms, these costs often dominate the process of factoring matrices stored in external memory \cite{adg2017,dgh2012,kc20}. Consequently, reducing communication overhead is critical for the efficient execution of any factoring algorithm . Algorithm \ref{alg:qr} achieves this by leveraging several parallelizable matrix-matrix multiplications. Moreover, during the QR decomposition, computations are restricted to small $m\times k$ matrices in each step. In \cite{dgh2012}, the proposed method can carry out QR decomposition with minimum communication costs. These features make Algorithm \ref{alg:A} well-suited for execution on high-performance computing platforms.

\section{Numerical experiments}\label{sec:numexp}

In this section, we compare our algorithm with the state-of-the-art image processing methods by applying them to image compression and reconstruction. We also compare our method with the recently proposed randomized methods on robustness analysis and low-rank approximation. The experiments were run in MATLAB R2021b on a desktop PC with a 3.30 GHz AMD Ryzen 9 5900HX processor and 16 GB of memory. The source code of our method is available  at \url{https://github.com/xuweiwei1/EOD-ABE.git}.

\subsection{Image compression and reconstruction}
The goal of this experiment is to evaluate the performance of Algorithm \ref{alg:A} on real world data.
We randomly choose two color images (``Baboon'' and ``Lighthouse'') from the real datasets\footnote{\url{https://www.imageprocessingplace.com/root_files_V3/image_databases.htm} and \url{http://r0k.us/graphics/kodak}}.
Algorithm \ref{alg:A} is  compared with the well-known  Economy-sized SVD, RSVD \cite{hmt11}, TSR-SVD \cite{hmt11}, CoR-UTV \cite{kl18}, RP-TSOD \cite{kc20}, PbP-QLP \cite{kc21}, randUTV \cite{mqh2019}, farPCA \cite{fy2023}, Adaptive PCA \cite{dyxl2020}, Adaptive RSVD \cite{hmt11} and FSVD \cite{nsr}.

\begin{figure}
  \centering
      \includegraphics[width=5.0in]{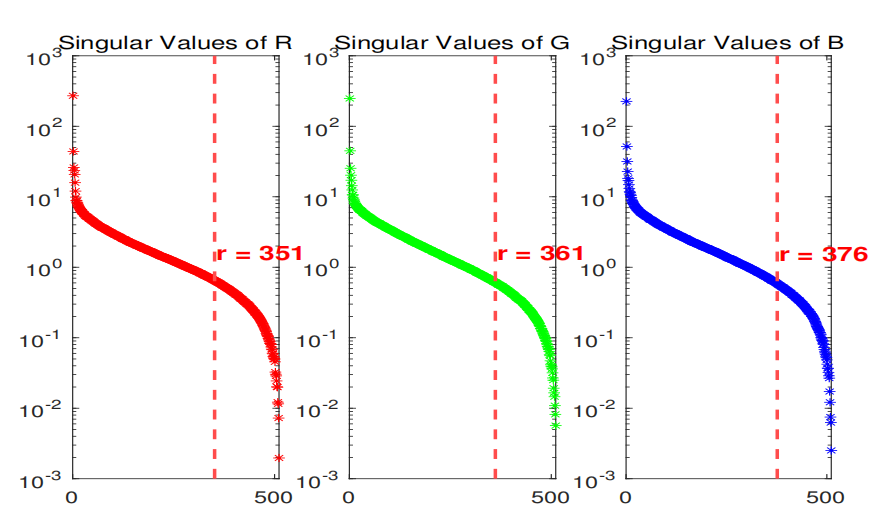}
  \caption{Singular values of the three color channels $R$, $G$ and $B$ of the color image ``Baboon''. To show this more clearly, the last singular values of the three matrices $R$, $G$, and $B$ are not drawn in the figure. The vertical red dotted lines in the figure show the numerical rank calculated by Algorithm \ref{alg:A}.}\label{rgb_sv}
\end{figure}

\begin{table}
\setlength\tabcolsep{3pt}
\renewcommand{\arraystretch}{1.25}  
	\centering \scriptsize
	\caption{Comparison of different methods on  color image ``Baboon'' with dimension $512\times512$}\label{example-t}
	\begin{tabular}{lccccccc}
		\hline
		\multirow{2}{*}{Algorithms}  & &  & Time (second)  & & \multirow{2}{*}{RelErr} & \multirow{2}{*}{PSNR (dB)} & \multirow{2}{*}{SSIM}\\
		  \cline{3-5}
		& & $R_{\text{time}}$  / rank & $G_{\text{time}}$ / rank & $B_{\text{time}}$  / rank& & \\
		\hline
		\multirow{1}{*}{Economy-sized SVD}   &               & 38.24 / 351 & 37.11 / 361 & 36.69 / 376 & 0.015881 & 41.724372 & 0.997211 \\
\cline{2-8}
		\multirow{3}{*}{RSVD \cite{hmt11}}        & $\tau=0$ & 28.91 / 351 & 31.04 / 361 & 33.70 / 376 & 0.036394 & 34.521194  & 0.985974 \\
		& $\tau=1$ & 40.21 / 341 & 42.86 / 362 & 46.27 / 380 & 0.018546 & 40.376575 & 0.996184 \\
		& $\tau=2$ & 50.97 / 352 & 54.14 / 371 & 58.48 / 385 & 0.015548 & 41.908470 & 0.997299 \\
\cline{2-8}
		\multirow{3}{*}{TSR-SVD \cite{hmt11}}          & $\tau=0$ & 34.44 / 351 & 36.81 / 361  & 39.70 / 376  & 0.036384 & 34.523664 & 0.986045 \\
		& $\tau=1$ & 54.92 / 341 & 58.97 / 362 & 64.40 / 380  & 0.018479 & 40.408426 & 0.996205 \\
		& $\tau=2$ & 75.61 / 352 & 81.72 / 371 & 88.67 / 385  & 0.015568 & 41.896972 & 0.997281 \\
\cline{2-8}
		\multirow{3}{*}{CoR-UTV \cite{kl18}}           & $\tau=0$ &\underline{23.63} / 351 & \underline{25.13} / 361 & \underline{27.03} / 376 & 0.036417 & 34.515798 & 0.986013 \\
		& $\tau=1$ & 33.85 / 341 & 36.34 / 362 & 39.14 / 380 & 0.018535 & 40.381701 & 0.996169 \\
		& $\tau=2$ & 44.43 / 352 & 47.67 / 371 & 51.37 / 385 & 0.015564 & 41.899396 & 0.997286 \\
\cline{2-8}
		\multirow{3}{*}{PbP-QLP \cite{kc20}}           & $\tau=0$ & 32.50 / 351 & 34.57 / 361 & 37.32 / 376 & 0.036603 & 34.471466 & 0.986532 \\
		& $\tau=1$ & 42.13 / 341 & 45.32 / 362 & 48.98 / 380 & 0.018493 & 40.401700 & 0.996210 \\
		& $\tau=2$ & 53.07 / 352 & 56.71 / 371 & 61.13 / 385 & 0.015575 & 41.893133 & 0.997275 \\
\cline{2-8}
		\multirow{3}{*}{RP-TSOD \cite{kc21}}           & $\tau=0$ & 25.20 / 351 & 26.76 / 361 & 28.80 / 376 & 0.036595 & 34.473419 & 0.986465 \\
		& $\tau=1$ & 35.23 / 341 & 37.59 / 362 & 40.59 / 380 & 0.018503 & 40.396742 & 0.996180 \\
		& $\tau=2$ & 45.69 / 352 & 48.91 / 371 & 52.78 / 385 & 0.015593 & 41.883123 & 0.997282 \\
\cline{2-8}
		\multirow{3}{*}{randUTV \cite{mqh2019}}           & $\tau=0$ & 48.92 / 351 & 48.29 / 361 & 48.24 / 376 & 0.037403 & 34.283740 & 0.986072 \\
		& $\tau=1$ & 50.27 / 341 & 50.16 / 362 & 50.57 / 380 & 0.017121 & 41.071143 & 0.996775 \\
		& $\tau=2$ & 52.58 / 352 & 52.74 / 371 & 52.92 / 385 & 0.016158 & 41.573854 & 0.997117 \\
\cline{2-8}
		\multirow{3}{*}{farPCA \cite{fy2023}}            & $\tau=0$ & 31.71 / 360 & 31.74 / 360 & 32.04 / 360 & 0.037381 & 34.288735 & 0.985270 \\
		& $\tau=1$ & 40.79 / 360 & 40.80 / 360 & 40.85 / 360 & 0.018508 & 40.394744 & 0.996222 \\
		& $\tau=2$ & 49.52 / 360 & 49.70 / 360 & 49.76 / 360 & 0.017108 & 41.077914 & 0.996797 \\
\cline{2-8}
		\multirow{3}{*}{Adaptive PCA \cite{dyxl2020}}      & $\tau=0$ & 29.45 / 360 & 29.29 / 360 & 29.22 / 360 & 0.037466 & 34.269017 & 0.985213 \\
		& $\tau=1$ & 36.20 / 360 & 36.03 / 360 & 36.05 / 360 & 0.018518 & 40.389887 & 0.996235 \\
		& $\tau=2$ & 40.62 / 360 & 40.51 / 360 & 40.71 / 360 & 0.017106 & 41.078877 & 0.996788 \\
\cline{2-8}
		\multirow{3}{*}{Adaptive RSVD \cite{hmt11}}      & $\tau=0$ & 41.17 / 348 & 47.37 / 362 & 52.92 / 369 & 0.036839 & 34.375644 & 0.985570 \\
		& $\tau=1$ & 68.32 / 362 & 56.83 / 352 & 69.63 / 375 & 0.017763 & 40.677953 & 0.996567 \\
		& $\tau=2$ & 71.54 / 347 & 67.06 / 353 & 70.90 / 381 & 0.015854 & 41.737632 & 0.997139 \\
\cline{2-8}
		\multirow{3}{*}{FSVD \cite{nsr}} & $\tau=0$ & 23.98 / 356 & 26.70 / 364 & 28.10 / 390 & 0.034276 & 35.029707 & 0.987661 \\
		& $\tau=1$ & 33.63 / 363 & 38.78 / 370 & 40.16 / 387 & 0.016032 & 41.628486 & 0.997172 \\
		& $\tau=2$ & 43.53 / 348 & 48.07 / 371 & 50.77 / 392 & 0.015226 & 41.934650 & 0.997301 \\
\cline{2-8}
		\multirow{3}{*}{Algorithm \ref{alg:A}} & $\tau=0$ & \textbf{22.38} / 351 & \textbf{23.36} / 361 & \textbf{24.62}  / 376 & 0.036414 & 34.516514 & 0.985997 \\
		& $\tau=1$ & 32.95 / 341 & 34.85 / 362 & 36.90 / 380 & 0.018515 & 40.391477 & 0.996179 \\
		& $\tau=2$ & 43.68 / 352 & 46.20 / 371 & 48.87 / 385 & 0.015553 & 41.905593 & 0.997290 \\
		\hline
	\end{tabular}
\end{table}

\begin{figure}[!ht]
    \centering
      \includegraphics[width=5.0in]{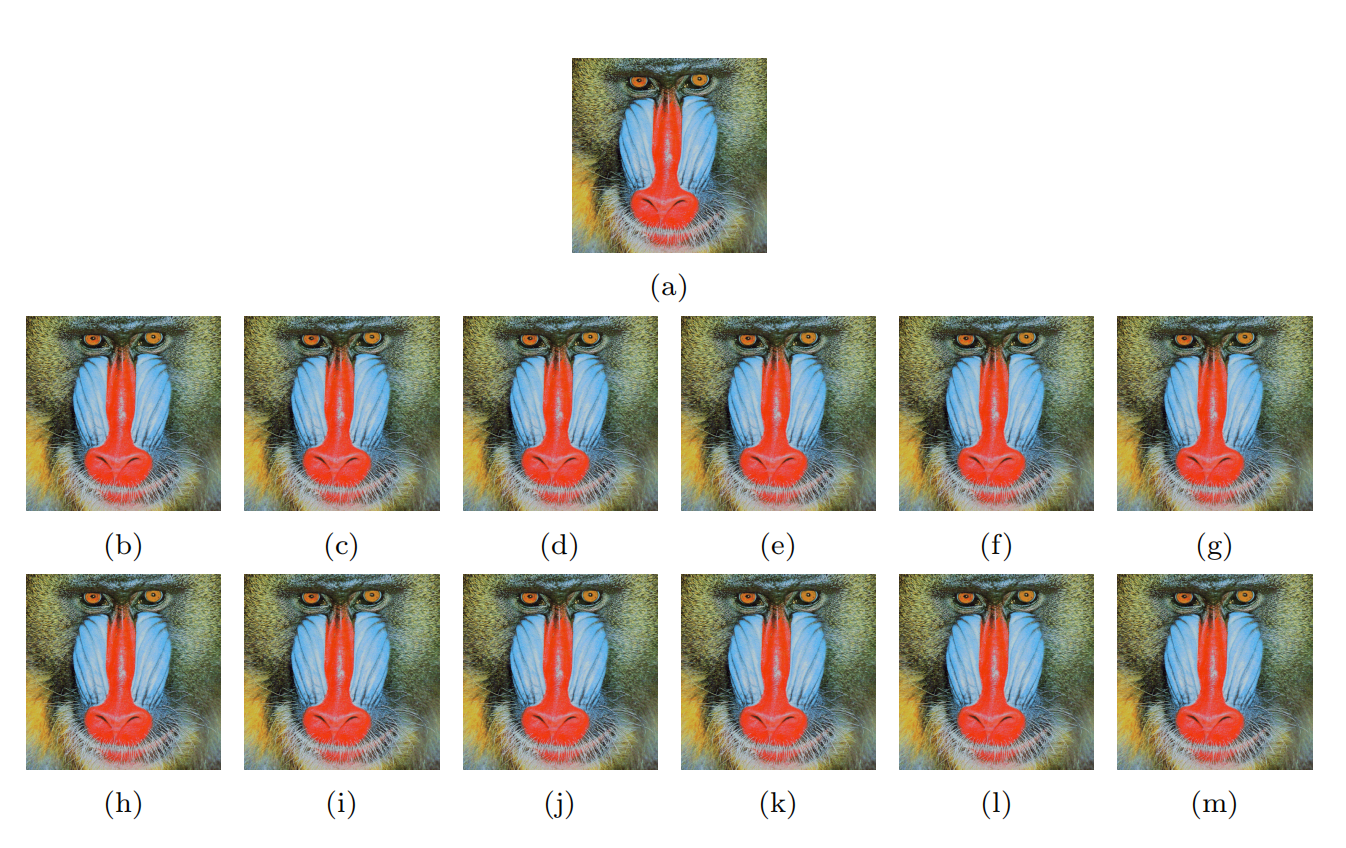}
  \caption{Low-rank compression and reconstruction of color image ``Baboon'' with $\tau = 0$ and  $\varepsilon = 0.001$ by
different methods. (a) Original. (b) Economy-sized SVD. (c) RSVD. (d) TSR-SVD. (e) CoR-UTV. (f) RP-TSOD. (g) PbP-QLP. (h) randUTV. (i) farPCA. (j) Adaptive PCA. (k) Adaptive RSVD. (l) FSVD. (m) Algorithm 2.}
  \label{example-figures1}
\end{figure}

\begin{figure}
    \centering
      \includegraphics[width=5.0in]{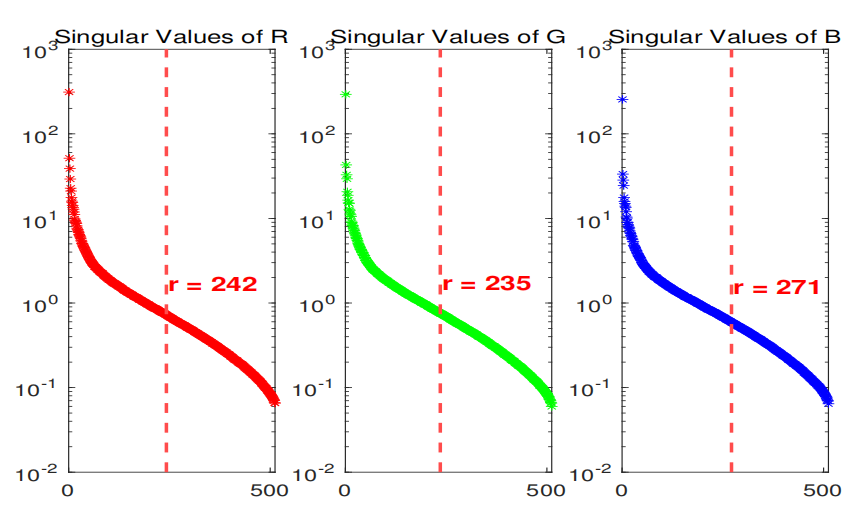}
  \caption{Singular values of the three color channels $R$, $G$ and $B$ of the color image ``Lighthouse''.  The vertical red dotted lines in the figure show the numerical rank calculated by Algorithm \ref{alg:A}.}\label{rgb_sv_kodim19}
\end{figure}

\begin{table}
\setlength\tabcolsep{3pt}
\renewcommand{\arraystretch}{1.25}  
	\centering \scriptsize
	\caption{Comparison of different methods for color image ``Lighthouse''  with dimension $768\times512$}\label{example-t2}
	\begin{tabular}{lccccccc}
		\hline
		\multirow{2}{*}{Algorithms}  & &  & Time (second)& & \multirow{2}{*}{RelErr} & \multirow{2}{*}{PSNR (dB)} & \multirow{2}{*}{SSIM}\\
		  \cline{3-5}
		& & $R_{\text{time}}$  / rank & $G_{\text{time}}$ / rank & $B_{\text{time}}$  / rank& & \\
		\hline
		\multirow{1}{*}{Economy-sized SVD}   &               & 56.22 	/	242	&	55.32 	/	235	&	56.49 	/	271	&	0.019407	&	40.600486	&	0.985618	\\
\cline{2-8}
		\multirow{3}{*}{RSVD \cite{hmt11}}        & $\tau=0$ & 19.02 	/	242	&	20.66 	/	235	&	23.68 	/	271	&	0.036228	&	35.201587	&	0.956165	\\
		& $\tau=1$ & 26.30 	/	247	&	28.79 	/	243	&	33.05 	/	280	&	0.020421	&	40.149823	&	0.984204	\\
		& $\tau=2$ & 34.14 	/	245	&	37.40 	/	244	&	42.95 	/	274	&	0.019438	&	40.600048	&	0.985430	\\
\cline{2-8}
		\multirow{3}{*}{TSR-SVD \cite{hmt11}}          & $\tau=0$ & 21.38 	/	242	&	23.12 	/	235	&	26.69 	/	271	&	0.036107	&	35.237366	&	0.956424	\\
		& $\tau=1$ & 36.39 	/	247	&	39.74 	/	243	&	46.21 	/	280	&	0.020348	&	40.181161	&	0.984292	\\
		& $\tau=2$ & 52.33 	/	245	&	57.12 	/	244	&	66.04 	/	274	&	0.019450	&	40.596928	&	0.985410	\\
\cline{2-8}
		\multirow{3}{*}{CoR-UTV \cite{kl18}}           & $\tau=0$ & \underline{17.06} 	/	242	&	 \underline{17.68} 	/	235	&	 \underline{21.98} 	/	271	&	0.036269	&	35.195809	&	0.955912	\\
		& $\tau=1$ & 24.56 	/	247	&	26.29 	/	243	&	31.44 	/	280	&	0.020393	&	40.161101	&	0.984227	\\
		& $\tau=2$ & 32.58 	/	245	&	34.80 	/	244	&	41.59 	/	274	&	0.019403	&	40.616214	&	0.985477	\\
\cline{2-8}
		\multirow{3}{*}{PbP-QLP \cite{kc20}}           & $\tau=0$ & 20.06 	/	242	&	21.52 	/	235	&	24.72 	/	271	&	0.036365	&	35.174096	&	0.957585	\\
		& $\tau=1$ & 27.65 	/	247	&	29.82 	/	243	&	34.41 	/	280	&	0.020397	&	40.155439	&	0.984330	\\
		& $\tau=2$ & 35.05 	/	245	&	38.44 	/	244	&	44.42 	/	274	&	0.019420	&	40.609231	&	0.985532	\\
\cline{2-8}
		\multirow{3}{*}{RP-TSOD \cite{kc21}}           & $\tau=0$ & 18.64 	/	242	&	20.02 	/	235	&	23.29 	/	271	&	0.036297	&	35.186467	&	0.957801	\\
		& $\tau=1$ & 26.28 	/	247	&	28.42 	/	243	&	32.97 	/	280	&	0.020391	&	40.160475	&	0.984388	\\
		& $\tau=2$ & 34.05 	/	245	&	36.89 	/	244	&	42.55 	/	274	&	0.019436	&	40.600323	&	0.985500	\\
\cline{2-8}
		\multirow{3}{*}{randUTV \cite{mqh2019}}           & $\tau=0$ & 82.58 	/	242	&	80.88 	/	235	&	80.74 	/	271	&	0.033128	&	35.992208	&	0.962736	\\
		& $\tau=1$ & 83.20 	/	247	&	83.47 	/	243	&	83.94 	/	280	&	0.019861	&	40.400429	&	0.985052	\\
		& $\tau=2$ & 87.10 	/	245	&	86.72 	/	244	&	87.22 	/	274	&	0.019444	&	40.583672	&	0.985574	\\
\cline{2-8}
		\multirow{3}{*}{farPCA \cite{fy2023}}            & $\tau=0$ & 20.13 	/	280	&	26.08 	/	280	&	44.18 	/	280	&	0.017381	&	40.386046	&	0.982038	\\
		& $\tau=1$ & 25.56 	/	280	&	25.97 	/	280	&	26.09 	/	280	&	0.017253	&	41.744839	&	0.987847	\\
		& $\tau=2$ & 31.46 	/	280	&	31.41 	/	280	&	31.37 	/	280	&	0.016211	&	42.285704	&	0.989065	\\
\cline{2-8}
		\multirow{3}{*}{Adaptive PCA \cite{dyxl2020}}      & $\tau=0$ & 24.53 	/	280	&	28.31 	/	280	&	42.15 	/	280	&	0.017421	&	40.374564	&	0.982141	\\
		& $\tau=1$ & 30.69 	/	280	&	30.66 	/	280	&	31.03 	/	280	&	0.017293	&	41.725847	&	0.987830	\\
		& $\tau=2$ & 34.62 	/	280	&	34.56 	/	280	&	34.49 	/	280	&	0.016212	&	42.285939	&	0.989056	\\
\cline{2-8}
		\multirow{3}{*}{Adaptive RSVD \cite{hmt11}}      & $\tau=0$ & 31.21 / 242 & 29.83 / 234 & 36.77 / 269 & 0.036500 & 35.141686 & 0.955918 \\
		& $\tau=1$ & 39.29 / 242 & 55.03 / 251 & 53.22 / 272 & 0.020585 & 40.115368 & 0.983733 \\
		& $\tau=2$ & 62.71 / 235 & 57.39 / 261 & 56.39 / 274 & 0.019135 & 40.707521 & 0.985338 \\
\cline{2-8}
		\multirow{3}{*}{FSVD \cite{nsr}}      & $\tau=0$ & 18.73 / 249 & 19.68 / 259 & 22.07 / 281 & 0.033482 & 35.894905 & 0.960981 \\
		& $\tau=1$ & 25.88 / 243 & 27.78 / 253 & 30.67 / 270 & 0.020527 & 40.155323 & 0.983889 \\
		& $\tau=2$ & 33.10 / 240 & 35.95 / 254 & 39.85 / 272 & 0.019290 & 40.670899 & 0.985369 \\
\cline{2-8}
		\multirow{3}{*}{Algorithm \ref{alg:A}} & $\tau=0$ & \textbf{15.88} 	/	242	&	\textbf{17.22} 	/	235	&	\textbf{19.97} 	/	271	&	0.036295	&	35.186443	&	0.956095	\\
		& $\tau=1$ & 23.21 	/	247	&	25.41	/	243	&	29.78 	/	280	&	0.020395	&	40.160224	&	0.984214	\\
		& $\tau=2$ & 30.98 	/	245	&	34.13 	/	244	&	39.65 	/	274	&	0.019421	&	40.608674	&	0.985465	\\
		\hline
	\end{tabular}
\end{table}

\begin{figure}[!ht]
   \centering
      \includegraphics[width=5.0in]{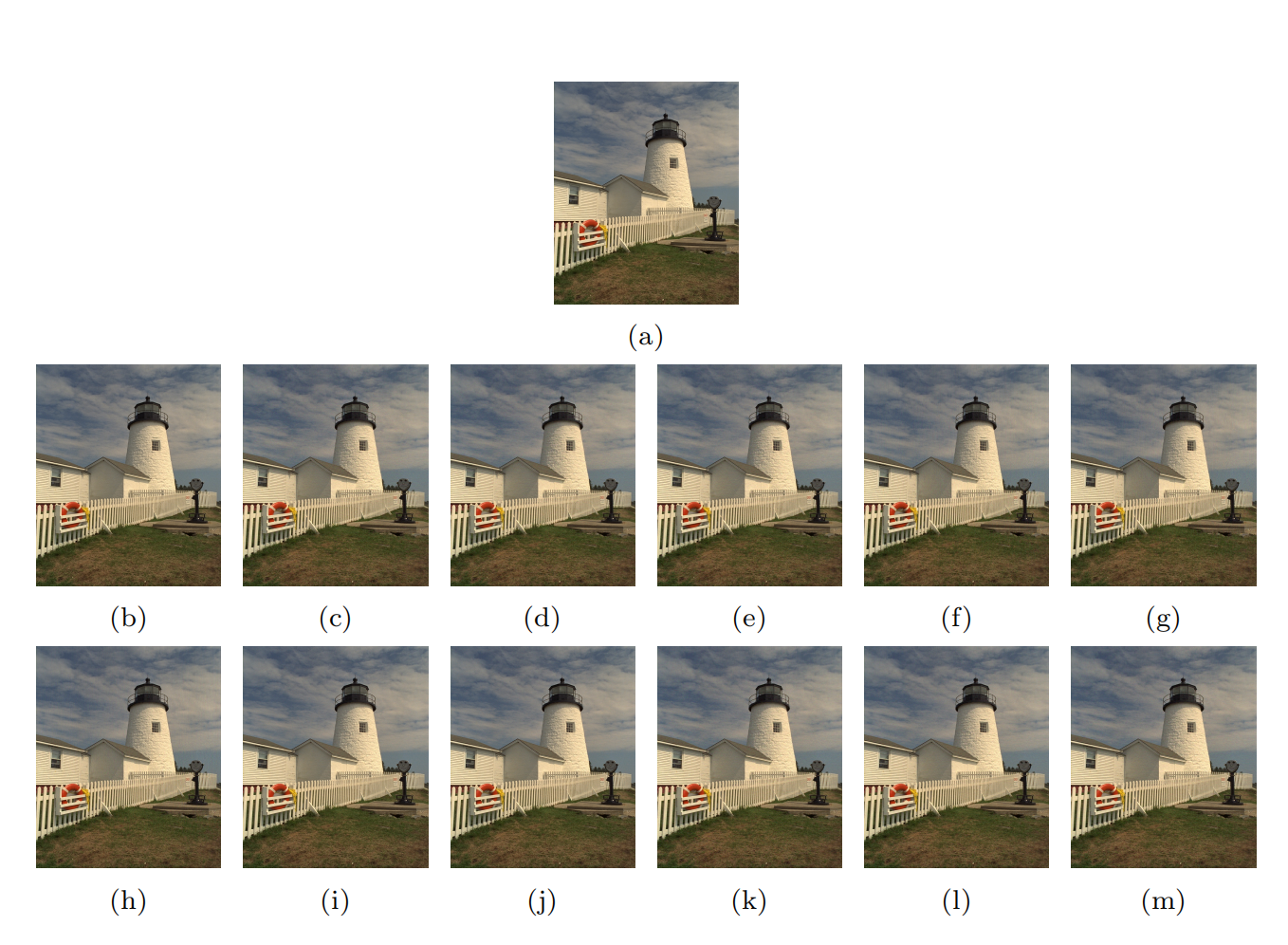}
  \caption{Low-rank compression and reconstruction of color image ``Lighthouse'' with $\tau = 0$ and $\varepsilon = 0.001$ by
different methods. (a) Original. (b) Economy-sized SVD. (c) RSVD. (d) TSR-SVD. (e) CoR-UTV. (f) RP-TSOD. (g) PbP-QLP. (h) randUTV. (i) farPCA. (j) Adaptive PCA. (k) Adaptive RSVD. (l) FSVD. (m) Algorithm 2.}
  \label{example-figures4}
\end{figure}

Let $A$ denote the original image. There are three standard criteria to evaluate the quality of the compressed image $\hat{A}$: 
\begin{enumerate}
\item  The peak signal-to-noise ratio value (PSNR), 
$${\rm PSNR}(\hat{A},A)=10\times\log_{10}\left(\frac{255^2n^2}{\|\hat{A}-A\|_F^2}\right).$$
\item  The structural similarity  index (SSIM), 
$${\rm SSIM}(\hat{A},A)=\frac{(4\mu_x\mu_y+c_1)(2\sigma_{xy}+c_2)}{(\mu_x^2+\mu_y^2+c_1)(\sigma_x^2+\sigma_y^2+c_2)},$$
where $x$ and $y$ are the vector forms of $\hat{A}$ and $A$,   $\mu_{x,y}$, $\sigma_{x,y}^2$ and $\sigma_{xy}$ respectively denote  their averages, variances and  covariances, and  $c_{1,2}$ are two constants.
  \item The relative error value (RelErr),
  \[\mathrm{RelErr}(\hat{A},A)=\frac{\|A-\hat{A}\|_{\rmF}}{\|A\|_{\rmF}}.\]
\end{enumerate}

The specific operation is as follows: During the experiment, we fix the $\varepsilon$ and compute the sampling parameter $d$ by Algorithm \ref{alg:A}. Such $d$ is substituted into the algorithms including from Economy-sized SVD to randUTV that require the matrix rank in advance. The compression results are shown in Figures \ref{example-figures1} and \ref{example-figures4}.
For each method, we compress these color images 1000 times and output the average result. 
Numerical results such as the calculation time and ranks  are shown in Tables \ref{example-t} and \ref{example-t2} with  the best values in bold and the second-best ones underlined.

From these numerical results, one can draw the following conclusions.
\begin{enumerate}
  \item  In Figures \ref{example-figures1} and \ref{example-figures4},  the low-rank approximations of all the compared methods recover the feature information of the original images.
      \item  In  Tables \ref{example-t} and \ref{example-t2},  the approximations computed by Algorithm  \ref{alg:A}  have almost the same PSNR and SSIM values with those by other algorithms. However, Algorithm \ref{alg:A}  costs less CPU time.  For the power factor $\tau=1$, the computational time of compressing images through Algorithm \ref{alg:A} is the shortest compared with other algorithms. 
      \item   Algorithm \ref{alg:A} can automatically reveal the correct numerical ranks (see Figures \ref{rgb_sv} and \ref{rgb_sv_kodim19}).  However,  the well-known RSVD \cite{hmt11}, TSR-SVD \cite{hmt11}, CoR-UTV \cite{kl18}, RP-TSOD \cite{kc20}, and PbP-QLP \cite{kc21} require prior knowledge of the matrix rank, and they cannot obtain strictly accurate compression when the rank is unknown.
  \item Since the rank obtained by farPCA and Adaptive PCA is always equal to an integer multiple of the block size, it is difficult to obtain the same rank as Algorithm \ref{alg:A}. Although they have a shorter calculation time when $\tau=2$, their relative error, PSNR and SSIM are not as good as Algorithm \ref{alg:A}.
For Adaptive RSVD and FSVD, they can adaptively obtain the approximate rank of the matrix, but they take a longer time than Algorithm \ref{alg:A}.
\end{enumerate}

\subsection{Dimensionality reduction for hyperspectral remote sensing}

In this experiment, we perform dimensionality reduction on hyperspectral image data and classify it using the k-nearest neighbors (KNN) algorithm with ten-fold cross-validation. We compare twelve dimensionality reduction algorithms: Economy-sized SVD, RSVD \cite{hmt11}, TSR-SVD \cite{hmt11}, CoR-UTV \cite{kl18}, RP-TSOD \cite{kc20}, PbP-QLP \cite{kc21}, randUTV \cite{mqh2019}, farPCA \cite{fy2023}, Adaptive PCA \cite{dyxl2020}, Adaptive RSVD \cite{hmt11}, FSVD \cite{nsr} and Algorithm \ref{alg:A}. 

The hyperspectral image data is derived from the Indian Pines dataset \cite{bbl2015}, which was captured by the Airborne Visible/Infrared Imaging Spectrometer (AVIRIS) in 1992, over a region of pine trees in Indiana, USA. A subset of the dataset with a spatial resolution of $145\times145$ pixels is selected and labeled for hyperspectral image classification (HSIC) purposes.
The dataset contains 220 spectral bands, but due to water absorption, bands 104-108, 150-163, and band 220 are excluded from the analysis. As a result, we use the remaining 200 bands for our experiments. 
We first reshape each spectral image of size $145\times145$ into a column vector of size $21025\times1$, and then concatenate all 200 column vectors to construct a data matrix of size $21025\times200$. Dimensionality reduction is then applied to this matrix using low-rank approximation techniques.
The ground truth available is designated into sixteen classes and is not all mutually exclusive. The dataset is publicly available and can be accessed at \url{https://purr.purdue.edu/publications/1947/1}.
Figure \ref{Indian_pines} shows the image cube of the Indian Pines Test Site data, where the three-dimensional structure of the hyperspectral cube highlights the spatial and spectral information contained within the dataset. Each slice of the cube represents a different spectral band, providing insight into the variations across wavelengths.

\begin{figure}[!t]
   \centering
      \includegraphics[width=5.0in]{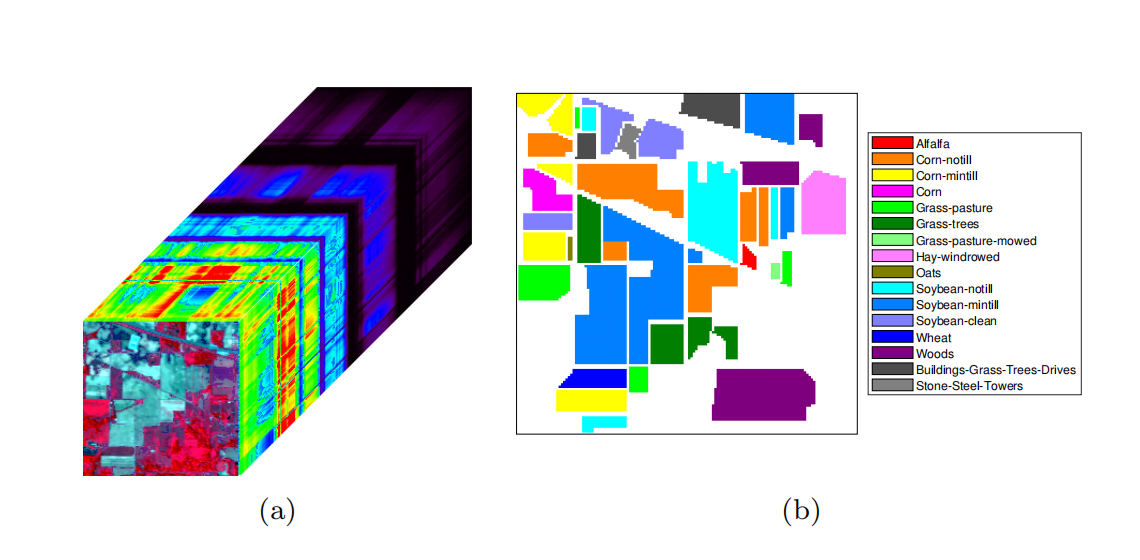}
  \caption{The Indian Pine Test Site data. (a) Image cube. (b) Hyperspectral image classification.}
  \label{Indian_pines}
\end{figure}

\begin{table}
\setlength\tabcolsep{3pt}
\renewcommand{\arraystretch}{1.25}  
  \centering \scriptsize
  \caption{Overall accuracy performance and time of different dimensionality reduction methods for the Indian Pine Test Site data. (The average of ten times, $\tau = 1$)}\label{table_Indian}
  \begin{tabular}{ccccccccc}
  \hline
  \multirow{2}{*}{Algorithms} & \multicolumn{2}{c}{$\varepsilon=0.1$, $N=3$} & \multicolumn{2}{c}{$\varepsilon=0.03$, $N=5$} & \multicolumn{2}{c}{$\varepsilon=0.02$, $N=9$} & \multicolumn{2}{c}{$\varepsilon=0.01$, $N=13$}\\
  \cline{2-9}
   & OA(\%) & Time(s) & OA(\%) & Time(s) & OA(\%) & Time(s) & OA(\%) & Time(s)\\
   \hline
  Economy-sized SVD              &	47.74 	&	0.0336 	&	64.24 	&	0.0311 	&	79.27 	&	0.0334 	&	93.90 	&	0.0354 	\\
RSVD \cite{hmt11}        &	50.89 	&	0.0055 	&	62.43 	&	\underline{0.0048} 	&	78.03 	&	0.0062 	&	95.72 	&	\underline{0.0077} 	\\
TSR-SVD \cite{hmt11}           &	49.43 	&	0.0062 	&	62.66 	&	0.0062 	&	77.82 	&	0.0079 	&	95.24 	&	0.0130 	\\
CoR-UTV \cite{kl18}            &	48.88 	&	0.0055 	&	63.22 	&	0.0061 	&	80.54 	&	0.0062 	&	95.57 	&	0.0094 	\\
RP-TSOD \cite{kc20}            &	52.96 	&	\underline{0.0048} 	&	62.34 	&	0.0060 	&	80.27 	&	0.0069 	&	95.49 	&	0.0109 	\\
PbP-QLP \cite{kc21}            &	49.77 	&	0.0054 	&	63.10 	&	0.0060 	&	80.18 	&	0.0063 	&	95.35 	&	0.0091 	\\
randUTV \cite{mqh2019}         &	55.78 	&	0.0124 	&	69.20 	&	0.0109 	&	88.78 	&	0.0143 	&	96.42 	&	0.0155 	\\
farPCA \cite{fy2023}           &	49.39 	&	0.0049 	&	63.02 	&	0.0050 	&	79.30 	&	\underline{0.0053} 	&	95.53 	&	0.0098 	\\
Adaptive PCA \cite{dyxl2020}   &	50.30 	&	0.0054 	&	62.58 	&	0.0061 	&	79.10 	&	0.0063 	&	95.01 	&	0.0114 	\\
Adaptive RSVD \cite{hmt11}     &	51.24 	&	0.0147 	&	66.97 	&	0.0140 	&	81.85 	&	0.0155 	&	94.20 	&	0.0207 	\\
FSVD \cite{nsr} &	 52.39 	&	0.0050  	&	71.21  	&	0.0053 	&	88.31  	&	0.0064  	&   95.81 	&	0.0099  	\\
Algorithm \ref{alg:A}          &	52.95 	&	\textbf{0.0033} 	&	71.41 	&	\textbf{0.0039} 	&	85.43 	&	\textbf{0.0042} 	&	95.41 	&	\textbf{0.0074} 	\\
  \hline
\end{tabular}
\end{table}

The numerical results, including the number of features ($N$), the overall accuracy (OA) and computational time, are presented in Table \ref{table_Indian} with  the best values in bold and  the second-best ones underlined.  We can observe significant performance variations across the compared  dimensionality reduction methods on the Indian Pine Test Site data. The OA is calculated using the formula:
\[OA=\frac{\sum_{i=1}^{C}n_{ii}}{\sum_{i=1}^{C}\sum_{j=1}^{C}n_{ij}},\]
where $C$ is the number of classes and $n_{ij}$ represents the number of samples of the $j$-th category classified into the $i$-th category.

In terms of overall accuracy, Algorithm \ref{alg:A} consistently outperforms most of the other techniques across various values of the parameter $\varepsilon$. Algorithm \ref{alg:A} shows a competitive balance between accuracy and computational efficiency. In particular, it exhibits the fastest processing time across all $\varepsilon$ values, making it highly efficient. For instance, at $\varepsilon=0.03$, the computational time for Algorithm \ref{alg:A} is 0.0039 seconds, which is the lowest among all methods, while still maintaining a relatively high accuracy.

Comparatively, methods such as Economy-sized SVD and RSVD also show strong performance in terms of accuracy, but their computational costs are noticeably higher than Algorithm \ref{alg:A}. Other methods like TSR-SVD and CoR-UTV offer reasonable accuracy but require significantly more time to compute, particularly as the number of features ($N$) increases.

\subsection{Robustness of Algorithm \ref{alg:A}} In this experiment, we use four classes of matrices from \cite{kc21} to illustrate the suitability and robustness of Algorithm \ref{alg:A}.  The first two classes contain one or multiple gaps in the spectrum and are particularly designed to investigate the rank-revealing property of Algorithm \ref{alg:A}. The second two classes have fast and slow decay singular values.  For the sake of simplicity, we focus on square matrices by setting the order $n = 1000$.
\begin{itemize}
  \item Matrix I (Low-rank plus noise). A rank-$r$ matrix with $r=20$ is formed as follows:
  \begin{equation}\label{matrix-1}
    A = U\Sigma V^{\rmT}+\alpha\sigma_{r} E,
  \end{equation}
  where $U\in \mathbb{R}^{n\times n}$ and $V \in \mathbb{R}^{n\times n}$ are orthogonal matrices,  $\Sigma = \mathrm{diag}\{\sigma_{1},$ $\cdots,$ $\sigma_{n}\}$ $ \in \mathbb{R}^{n\times n}$ is a diagonal matrix,   $\sigma_1, \ldots, \sigma_r$  decrease
linearly from 1 to $10^{-25}$,  $\sigma_{r+1}=\cdots=\sigma_{n}=0$, and $E\in\mathbb{R}^{n\times n}$ is a standard Gaussian matrix. 
  \begin{description}
    \item (i) When $\alpha=0.005$, $A$ has a gap $\approx 200$.
    \item (ii) When $\alpha=0.02$,  $A$ has a gap $\approx 50$.
  \end{description}
  \item Matrix II (The devil's stairs). This challenging matrix has multiple gaps in its spectrum. The singular values are arranged analogous to a descending staircase with each step consisting of 15 equal singular values. The singular values decay in steps, and are given by
\[\sigma_i = 10^{-\frac{4}{5} \cdot \left\lfloor \frac{i-1}{15} \right\rfloor}, \quad i=1,2,\cdots, n.\]
  \item Matrix III (Fast decay). This matrix is formed as follows:
  \begin{equation}\label{matrix-3}
    A = U\Sigma V^{\rmT},
  \end{equation}
  where the diagonal elements of $\Sigma$ have the form $\sigma_{i}=e^{-i/6}$, for $i=1,...,n$.
  \item Matrix IV (Slow decay). This matrix is also formed as Matrix III, but the diagonal elements of $\Sigma$ take the form $\sigma_{i}=i^{-2}$, for $i=1,...,n$.
\end{itemize}

\begin{figure}[!ht]
   \centering
      \includegraphics[width=5.0in]{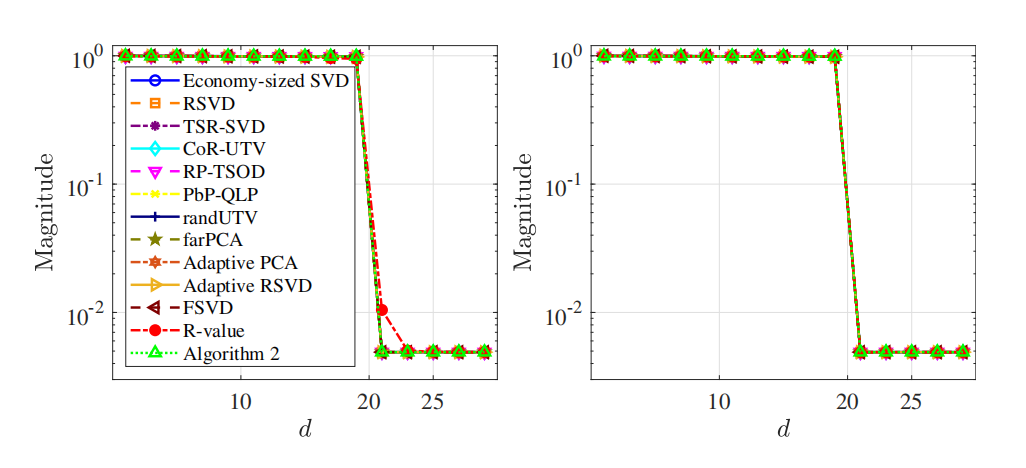}
  \caption{Singular value approximations for Matrix I with $\alpha=0.005$. Left: $\tau=0$. Right: $\tau=2$.}\label{plus_noiseLargeGap}
\end{figure}
\begin{figure}[!ht]
    \centering
      \includegraphics[width=5.0in]{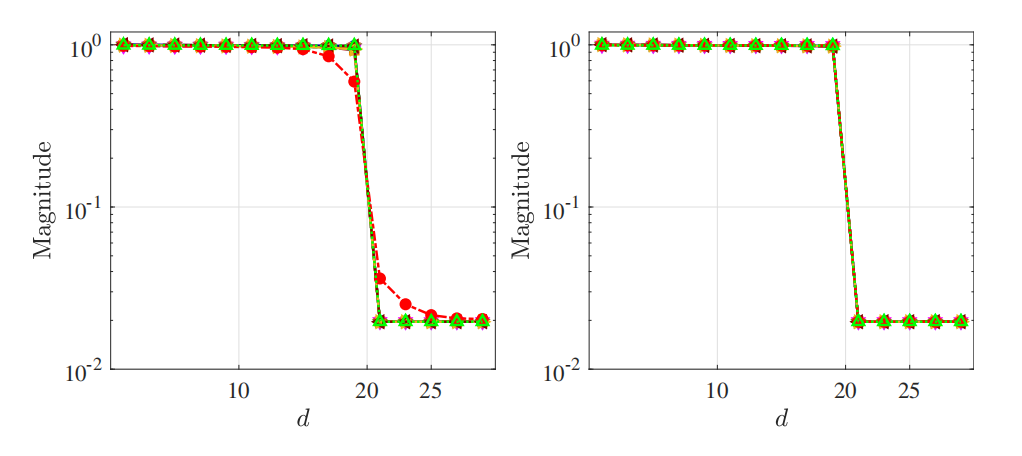}
  \caption{Singular value approximations for Matrix I with $\alpha=0.02$. Left: $\tau=0$. Right: $\tau=2$.}\label{plus_noiseSmallGap}
\end{figure}
\begin{figure}[!ht]
   \centering
      \includegraphics[width=5.0in]{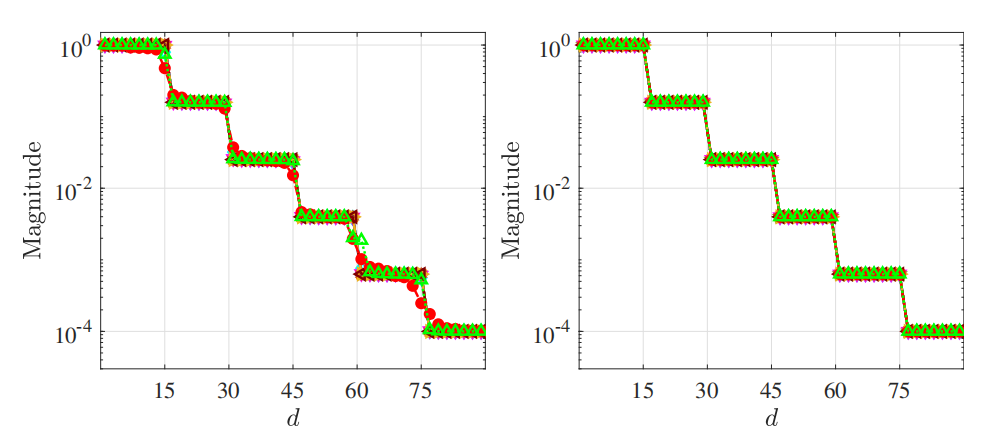}
  \caption{Singular value approximations for Matrix II. Left: $\tau=0$. Right: $\tau=2$.}\label{stairs}
\end{figure}
\begin{figure}[!ht]
    \centering
      \includegraphics[width=5.0in]{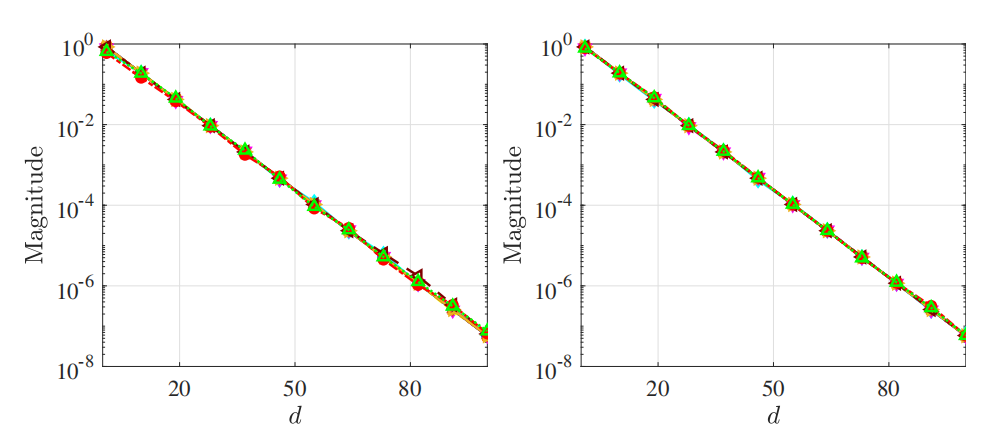}
  \caption{Singular value approximations for Matrix III. Left: $\tau=0$. Right: $\tau=2$.}\label{fast_decay}
\end{figure}
\begin{figure}[!ht]
   \centering
      \includegraphics[width=5.0in]{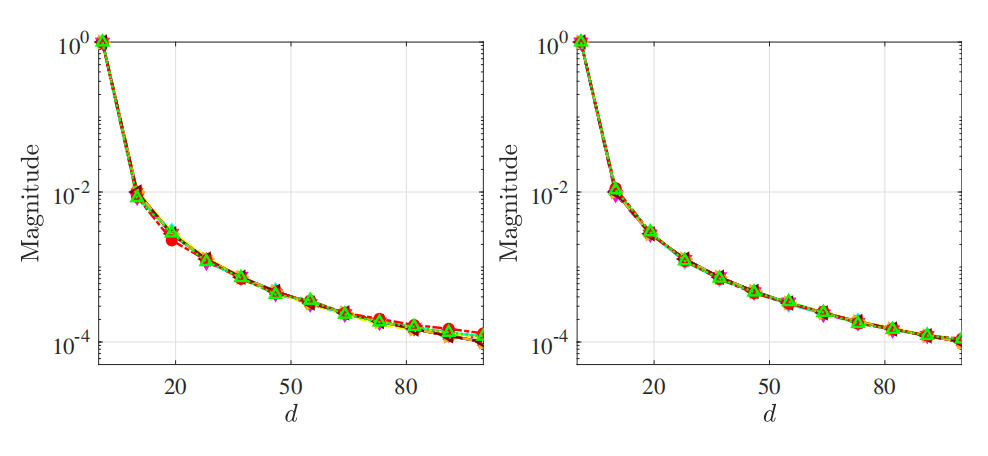}
  \caption{Singular value approximations for Matrix IV. Left: $\tau=0$. Right: $\tau=2$.}\label{slow_decay}
\end{figure}

The results for singular values estimation are plotted in Figures \ref{plus_noiseLargeGap}-\ref{slow_decay}. We make several observations:
\begin{enumerate}
  \item The numerical ranks of Matrix I with $\alpha=0.005$ and Matrix I with $\alpha=0.02$ are strongly revealed in $\Sigma$ generated by Algorithm \ref{alg:A} with $\tau=0$. This is due to the fact that the gaps in the spectra of these matrices are well-defined. The diagonal elements of $\Sigma$ generated by Algorithm \ref{alg:A} with subspace iteration technique are as accurate as those of the optimal SVD. Figures \ref{plus_noiseLargeGap} and \ref{plus_noiseSmallGap} show that Algorithm \ref{alg:A} is a rank-revealer.
  \item For Matrix II, the R-values of Algorithm \ref{alg:A} with $\tau=0$ do not clearly disclose the gaps in the matrix's spectrum. This is because the gaps are not substantial. Fortunately, Algorithm \ref{alg:A} with $\tau=0$ strongly reveals the gaps, which shows that the procedure leading to the formation of the upper triangular matrix $R$ provides a good first step for Algorithm \ref{alg:A}. The R-values of Algorithm \ref{alg:A} with subspace iteration clearly disclose the gaps.
  \item For Matrix III and Matrix IV, Algorithm \ref{alg:A} provides highly accurate singular values, showing similar performance as Economy-sized SVD, RSVD in \cite{hmt11}, TSR-SVD in \cite{hmt11}, CoR-UTV in \cite{kl18}, RP-TSOD in \cite{kc20} and PbP-QLP in \cite{kc21}.
\end{enumerate}

The results presented in Figures \ref{plus_noiseLargeGap}-\ref{slow_decay} demonstrate the applicability of Algorithm \ref{alg:A} in accurately estimating the singular values of matrices of different classes.

\subsection{Sensitivity to the truncation tolerance \(\varepsilon\)}
In this part, we use Matrix II: The devil's stairs from Section 4.3, which has size 90, to illustrate how the estimated rank and relative error vary with the truncation tolerance. The results are shown in Figure \ref{effect}.

\begin{figure}[!ht]
  \centering
      \includegraphics[width=5.0in]{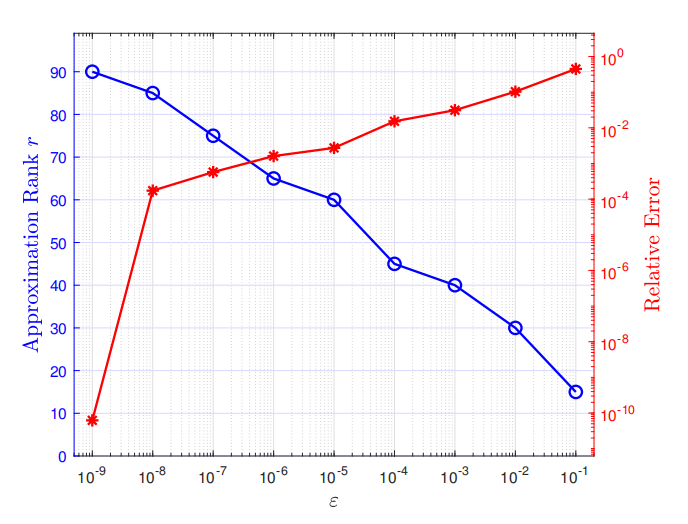}
  \caption{Effect of $\varepsilon$ on approximation rank and relative error.}\label{effect}
\end{figure}

Based on the truncated experiment results shown in the figure, the influence of the selection strategy for $\varepsilon$ on the approximation rank and the relative error can be clearly described:
As the tolerance $\varepsilon$ increases, the approximation rank $r$ decreases monotonically, while the relative error increases monotonically. Specifically:
\begin{itemize}
  \item When $\varepsilon$ is small (e.g., $10^{-9}$), the algorithm retains more principal components, resulting in a higher approximation rank $r$ and high approximation accuracy, with the relative error reaching as low as the $10^{-10}$ level.
  \item As the tolerance $\varepsilon$ increases from $10^{-9}$ to $10^{-8}$, the estimated rank decreases from 90 to 85. Since the true rank is exactly 90, this indicates that 5 nonzero singular values are discarded. Consequently, the approximation incurs a noticeable loss, with the relative error roughly on the order of $\sqrt{5} \times 10^{-4}$.
  \item Further increasing $\varepsilon$ (e.g., up to $10^{-1}$) leads to more aggressive truncation, retaining only a few principal components. This significantly reduces the approximation rank $r$ and improves computational efficiency.
\end{itemize}
This illustrates the relationship between the tolerance $\varepsilon$, the approximation rank $r$, and the relative error. By adjusting $\varepsilon$, one can control how many singular values are retained and thus select an appropriate trade-off tailored to specific accuracy and efficiency requirements. In practical batch-data applications, it is often effective to start with a relatively large $\varepsilon$ and gradually decrease it based on accuracy needs, thereby identifying an optimal balance point.

\subsection{Low-rank approximation}
We compare the speed and accuracy of Algorithm \ref{alg:A} with the existing state-of-the-art algorithms, in factoring input matrices with various dimensions.
We construct a new class of strictly rank-deficient matrices
$A \in \mathbb{R}^{n\times n}$ with rank $r$ of the form:
\begin{eqnarray}\label{mat2}
A = U\left(
         \begin{array}{cc}
           \Sigma & O_{r\times (n-r)} \\
           O_{(n-r)\times r} & O_{(n-r)\times (n-r)} \\
         \end{array}
       \right)
 V^{\rmT},
\end{eqnarray}
where $U\in \mathbb{R}^{n\times n}$ and $V \in \mathbb{R}^{n\times n}$ are orthogonal matrices, and $\Sigma=\mathrm{diag}\{\sigma_{1},\ldots,\sigma_{r}\} \in \mathbb{R}^{r\times r}$ is a diagonal matrix, $\sigma_{1}\geq\sigma_{2}\geq\cdots\geq\sigma_{r}>0$ are the $r$ non-zero singular values of $A$.
We generate the exact $\Sigma=\mathrm{diag}\{\sigma_{1},\ldots,\sigma_{r}\}$ by using the built-in function \texttt{rand(1,r)}
in MATLAB such that
\[\sigma_{1}\geq\sigma_{2}\geq\cdots\geq\sigma_{r}>0\]
and the orthogonal matrices $U$ and $V$ are generated by code \texttt{orth(randn(n,n))} in MATLAB.

Now we take $A$ generated by \eqref{mat2} as testing matrices in two cases.
First, we consider that when the rank of the matrix is unknown, the sampling size parameter $d$ may have the following three cases: 
$d = r$,
$d > r$ and
$d < r$.
To evaluate these scenarios, we used square matrices of various sizes, setting the matrix rank at $r = 0.4n$.
Sampling size parameters of $d = 0.4n$, $d = 0.6n$, and $d = 0.35n$ are tested.
The computation time and accuracy of Algorithm \ref{alg:A} are compared with the current state-of-the-art algorithms in the literature.  The results are presented in Tables \ref{0.4n}-\ref{0.35n}.

\begin{table}
\setlength\tabcolsep{3pt}
\renewcommand{\arraystretch}{1.25}  
  \centering\scriptsize
  \caption{Numerical results of low-rank approximation with different methods for matrix \eqref{mat2} ($r=0.4n$, $d=0.4n$)}\label{0.4n}
  \begin{tabular}{lccccccc}
    \hline
    \multirow{2}{*}{Algorithms}  & & \multicolumn{2}{c}{$n=4000$} & \multicolumn{2}{c}{$n=8000$} & \multicolumn{2}{c}{$n=12000$} \\
    \cline{3-8}
     & & Time(s) & RelErr& Time(s) & RelErr& Time(s) & RelErr  \\
    \hline
    \multirow{1}{*}{Economy-sized SVD} &  & 10.2 & 4.9E-15 & 73.2 & 6.2E-15& 244.8 & 7.0E-15\\
    \cline{2-8}
    \multirow{3}{*}{RSVD \cite{hmt11}}	&	$\tau=0$	&	1.8 	&	2.7E-14	&	18.6 	&	6.6E-14	&	62.5 	&	8.2E-14	\\
	                                &	$\tau=1$	&	2.7 	&	4.0E-15	&	25.2 	&	4.9E-15	&	85.7 	&	5.5E-15	\\
	                                &	$\tau=2$	&	3.6 	&	4.1E-15	&	33.0 	&	5.0E-15	&	109.8 	&	5.8E-15	\\
\cline{2-8}
\multirow{3}{*}{TSR-SVD \cite{hmt11}}	&	$\tau=0$	&	2.0 	&	1.4E-13	&	20.4 	&	2.5E-13	&	56.0 	&	4.3E-14	\\
	                            &	$\tau=1$	&	4.0 	&	4.1E-15	&	35.4 	&	4.9E-15	&	102.9 	&	8.1E-15	\\
	                            &	$\tau=2$	&	5.8 	&	4.1E-15	&	50.3 	&	5.0E-15	&	149.9 	&	8.0E-15	\\
\cline{2-8}
\multirow{3}{*}{CoR-UTV \cite{kl18}}	&	$\tau=0$	&	1.7 	&	9.2E-13	&	14.0 	&	1.1E-11	&	45.5 	&	1.1E-11	\\
	                            &	$\tau=1$	&	2.7 	&	1.8E-15	&	21.2 	&	2.1E-15	&	68.8 	&	2.3E-15	\\
	                            &	$\tau=2$	&	3.6 	&	1.8E-15	&	28.9 	&	2.0E-15	&	92.0 	&	2.3E-15	\\
\cline{2-8}
\multirow{3}{*}{RP-TSOD \cite{kc20}}	&	$\tau=0$	&	2.4 	&	3.8E-14	&	21.2 	&	1.2E-13	&	71.0 	&	9.1E-14	\\
	                            &	$\tau=1$	&	3.4 	&	1.7E-15	&	28.4 	&	2.1E-15	&	93.8 	&	2.3E-15	\\
	                            &	$\tau=2$	&	4.3 	&	1.6E-15	&	35.9 	&	1.9E-15	&	117.2 	&	2.2E-15	\\
\cline{2-8}
\multirow{3}{*}{PbP-QLP \cite{kc21}}	&	$\tau=0$	&	1.7 	&	4.7E-14	&	18.5 	&	7.7E-14	&	64.2 	&	7.3E-14	\\
	                            &	$\tau=1$	&	2.7 	&	1.4E-15	&	26.0 	&	1.4E-15	&	90.7 	&	1.4E-15	\\
	                            &	$\tau=2$	&	3.6 	&	1.3E-15	&	33.1 	&	1.4E-15	&	113.6 	&	1.4E-15	\\
\cline{2-8}
\multirow{3}{*}{randUTV \cite{mqh2019}}	&	$\tau=0$	&	7.0 	&	3.6E-15	&	52.1 	&	4.1E-15	&	210.7 	&	4.1E-15	\\
	                            &	$\tau=1$	&	7.8 	&	3.7E-15	&	56.2 	&	3.9E-15	&	213.5 	&	4.0E-15	\\
	                            &	$\tau=2$	&	8.1 	&	3.7E-15	&	60.3 	&	3.9E-15	&	217.6 	&	4.1E-15	\\
\cline{2-8}
\multirow{3}{*}{farPCA \cite{fy2023}}	    &	$\tau=0$	&	4.9 	&	3.4E-11	&	43.7 	&	1.1E-10	&	89.9 	&	3.0E-10	\\
	                            &	$\tau=1$	&	8.1 	&	1.8E-13	&	64.9 	&	4.9E-11	&	221.0 	&	2.9E-11	\\
	                            &	$\tau=2$	&	10.5 	&	1.8E-12	&	76.7 	&	3.0E-11	&	273.3 	&	1.8E-11	\\
\cline{2-8}
\multirow{3}{*}{Adaptive PCA \cite{dyxl2020}}	    &	$\tau=0$	&	2.8 	&	5.6E-14	&	22.6 	&	9.2E-14	&	52.9 	&	2.1E-13	\\
	                            &	$\tau=1$	&	3.8 	&	1.5E-14	&	30.4 	&	4.2E-14	&	72.9 	&	3.1E-13	\\
	                            &	$\tau=2$	&	4.5 	&	6.3E-14	&	36.2 	&	1.6E-13	&	88.5 	&	5.5E-13	\\
\cline{2-8}
\multirow{3}{*}{Adaptive RSVD \cite{hmt11}}	    &	$\tau=0$	&	5.5 	&	6.2E-14	&	56.3 	&	1.8E-13	&	188.6 	&	2.5E-13	\\
	                            &	$\tau=1$	&	6.5 	&	4.2E-15	&	64.2 	&	4.9E-15	&	212.2 	&	8.0E-15	\\
	                            &	$\tau=2$	&	7.5 	&	4.1E-15	&	71.8 	&	8.0E-15	&	236.1 	&	8.0E-15	\\
\cline{2-8}
\multirow{3}{*}{FSVD \cite{nsr}}	    &	$\tau=0$	&	\underline{1.6} 	&	9.2E-12	&	\underline{13.4} 	&	2.8E-11	&	\underline{43.9} 	&	4.3E-11	\\
	                            &	$\tau=1$	&	2.5 	&	2.8E-15	&	20.6 	&	3.3E-15	&	67.6 	&	3.7E-15	\\
	                            &	$\tau=2$	&	3.5 	&	2.8E-15	&	27.6 	&	3.3E-15	&	91.4 	&	3.6E-15	\\
\cline{2-8}
\multirow{3}{*}{Algorithm \ref{alg:A}}	&	$\tau=0$	&	\textbf{1.5} 	&	3.1E-13	&	\textbf{10.6} 	&	1.1E-12	&	\textbf{29.3} 	&	9.4E-12	\\
	                                        &	$\tau=1$	&	2.3 	&	1.3E-15	&	18.0 	&	1.3E-15	&	52.5 	&	1.3E-15	\\
	                                        &	$\tau=2$	&	3.3 	&	1.2E-15	&	25.6 	&	1.3E-15	&	75.5 	&	1.3E-15	\\
    \hline
  \end{tabular}
\end{table}

\begin{table}
\setlength\tabcolsep{3pt}
\renewcommand{\arraystretch}{1.25}  
  \centering\scriptsize
  \caption{Numerical results of low-rank approximation with different methods for matrix \eqref{mat2} ($r=0.4n$, $d=0.6n$)}\label{0.6n}
  \begin{tabular}{lcrrrrrr}
    \hline
    \multirow{2}{*}{Algorithms}  & & \multicolumn{2}{c}{$n=4000$} & \multicolumn{2}{c}{$n=8000$} & \multicolumn{2}{c}{$n=12000$} \\
    \cline{3-8}
     & & Time(s) & RelErr& Time(s) & RelErr& Time(s) & RelErr  \\
    \hline
    \multirow{1}{*}{Economy-sized SVD} &  & 10.2 & 4.9E-15 & 77.3 & 6.1E-15& 244.2 & 6.9E-15\\
\cline{2-8}
    \multirow{3}{*}{RSVD \cite{hmt11}}	&	$\tau=0$	&	4.9 	&	4.0E-15	&	48.7 	&	4.8E-15	&	162.8 	&	5.8E-15	\\
	                                &	$\tau=1$	&	6.4 	&	4.3E-15	&	61.3 	&	5.2E-15	&	202.9 	&	6.2E-15	\\
	                                &	$\tau=2$	&	8.2 	&	4.2E-15	&	74.2 	&	5.2E-15	&	243.2 	&	6.2E-15	\\
\cline{2-8}
\multirow{3}{*}{TSR-SVD \cite{hmt11}}	&	$\tau=0$	&	5.6 	&	4.1E-15	&	38.2 	&	6.7E-15	&	115.7 	&	7.6E-15	\\
	                            &	$\tau=1$	&	8.6 	&	4.3E-15	&	64.1 	&	6.7E-15	&	193.8 	&	7.7E-15	\\
	                            &	$\tau=2$	&	12.2 	&	4.5E-15	&	88.8 	&	6.9E-15	&	272.4 	&	7.7E-15	\\
\cline{2-8}
\multirow{3}{*}{CoR-UTV \cite{kl18}}	&	$\tau=0$	&	3.7 	&	5.1E-13	&	30.3 	&	5.3E-14	&	98.6 	&	1.7E-13	\\
	                            &	$\tau=1$	&	5.4 	&	1.8E-15	&	43.0 	&	2.1E-15	&	138.4 	&	2.4E-15	\\
	                            &	$\tau=2$	&	6.9 	&	1.8E-15	&	55.8 	&	2.1E-15	&	182.5 	&	2.3E-15	\\
\cline{2-8}
\multirow{3}{*}{RP-TSOD \cite{kc20}}	&	$\tau=0$	&	5.4 	&	1.9E-15	&	46.5 	&	2.2E-15	&	157.9 	&	2.4E-15	\\
	                            &	$\tau=1$	&	7.0 	&	1.7E-15	&	59.1 	&	2.1E-15	&	195.9 	&	2.4E-15	\\
	                            &	$\tau=2$	&	8.6 	&	1.7E-15	&	71.3 	&	2.0E-15	&	233.1 	&	2.2E-15	\\
\cline{2-8}
\multirow{3}{*}{PbP-QLP \cite{kc21}}	&	$\tau=0$	&	4.5 	&	1.6E-15	&	46.4 	&	1.6E-15	&	150.9 	&	1.6E-15	\\
	                            &	$\tau=1$	&	5.4 	&	1.4E-15	&	53.7 	&	1.4E-15	&	173.5 	&	1.4E-15	\\
	                            &	$\tau=2$	&	6.4 	&	1.4E-15	&	66.5 	&	1.4E-15	&	195.7 	&	1.4E-15	\\
\cline{2-8}
\multirow{3}{*}{randUTV \cite{mqh2019}}	&	$\tau=0$	&	7.2 	&	3.6E-15	&	52.3 	&	4.0E-15	&	211.1 	&	4.2E-15	\\
	                            &	$\tau=1$	&	7.7 	&	3.7E-15	&	56.8 	&	3.8E-15	&	216.1 	&	4.0E-15	\\
	                            &	$\tau=2$	&	8.4 	&	3.7E-15	&	60.2 	&	3.8E-15	&	219.6 	&	4.0E-15	\\
\cline{2-8}
\multirow{3}{*}{farPCA \cite{fy2023}}	    &	$\tau=0$	&	4.8 	&	7.7E-10	&	40.5 	&	2.9E-10	&	138.2 	&	5.5E-09	\\
	                            &	$\tau=1$	&	8.2 	&	1.2E-11	&	62.4 	&	3.2E-12	&	217.9 	&	2.2E-10	\\
	                            &	$\tau=2$	&	12.1 	&	1.6E-11	&	74.1 	&	4.4E-12	&	164.1 	&	1.7E-12	\\
\cline{2-8}
\multirow{3}{*}{Adaptive PCA \cite{dyxl2020}}	    &	$\tau=0$	&	3.1 	&	2.8E-14	&	22.1 	&	1.1E-13	&	52.2 	&	6.1E-14	\\
	                            &	$\tau=1$	&	4.3 	&	5.2E-14	&	29.7 	&	6.2E-13	&	70.8 	&	9.4E-14	\\
	                            &	$\tau=2$	&	5.3 	&	2.9E-12	&	35.4 	&	5.1E-14	&	86.9 	&	7.6E-13	\\
\cline{2-8}
\multirow{3}{*}{Adaptive RSVD \cite{hmt11}}	    &	$\tau=0$	&	5.5 	&	1.3E-13	&	56.4 	&	6.6E-14	&	187.9 	&	1.9E-13	\\
	                            &	$\tau=1$	&	6.5 	&	4.0E-15	&	64.2 	&	4.9E-15	&	217.2 	&	7.9E-15	\\
	                            &	$\tau=2$	&	7.4 	&	4.2E-15	&	71.7 	&	5.0E-15	&	234.2 	&	8.0E-15	\\
\cline{2-8}
\multirow{3}{*}{FSVD \cite{nsr}}	    &	$\tau=0$	&	\underline{1.8} 	&	1.4E-12	&	\underline{14.2} 	&	3.9E-11	&	\underline{45.9} 	&	3.6E-11	\\
	                            &	$\tau=1$	&	2.8 	&	2.8E-15	&	21.6 	&	3.2E-15	&	69.6 	&	3.6E-15	\\
	                            &	$\tau=2$	&	3.7 	&	2.8E-15	&	29.2 	&	3.3E-15	&	93.4 	&	3.7E-15	\\
\cline{2-8}
\multirow{3}{*}{Algorithm \ref{alg:A}}	&	$\tau=0$	&	\textbf{1.5} 	&	2.5E-13	&	\textbf{10.5} 	&	8.2E-12	&	\textbf{28.5} 	&	7.4E-12	\\
	                                        &	$\tau=1$	&	2.4 	&	1.3E-15	&	17.8 	&	1.3E-15	&	51.1	&	1.3E-15	\\
	                                        &	$\tau=2$	&	3.2 	&	1.2E-15	&	25.1 	&	1.3E-15	&	73.6 	&	1.3E-15	\\
    \hline
  \end{tabular}
\end{table}

\begin{table}
\setlength\tabcolsep{3pt}
\renewcommand{\arraystretch}{1.25}  
  \centering\scriptsize
  \caption{Numerical results of low-rank approximation with different methods for matrix \eqref{mat2} ($r=0.4n$, $d=0.35n$)}\label{0.35n}
  \begin{tabular}{lcrrrrrr}
    \hline
    \multirow{2}{*}{Algorithms}  & & \multicolumn{2}{c}{$n=4000$} & \multicolumn{2}{c}{$n=8000$} & \multicolumn{2}{c}{$n=12000$} \\
    \cline{3-8}
     & & Time(s) & RelErr& Time(s) & RelErr& Time(s) & RelErr  \\
    \hline
    \multirow{1}{*}{Economy-sized SVD} &  & 10.2 & 4.9E-15 & 72.9 & 6.2E-15& 243.6 & 7.0E-15\\
\cline{2-8}
    \multirow{3}{*}{RSVD \cite{hmt11}}	&	$\tau=0$	&	\underline{1.4} 	&	1.1E-01	&	12.9 	&	1.0E-01	&	44.8 	&	1.0E-01	\\
&	$\tau=1$	&	2.2 	&	3.9E-02	&	18.9 	&	3.9E-02	&	64.1 	&	3.8E-02	\\
&	$\tau=2$	&	3.0 	&	3.5E-02	&	25.4 	&	3.5E-02	&	83.2 	&	3.4E-02	\\
\cline{2-8}
\multirow{3}{*}{TSR-SVD \cite{hmt11}}	&	$\tau=0$	&	1.7 	&	1.1E-01	&	14.8 	&	1.0E-01	&	42.6 	&	1.0E-01	\\
&	$\tau=1$	&	3.3 	&	3.9E-02	&	27.4 	&	3.9E-02	&	81.5 	&	3.8E-02	\\
&	$\tau=2$	&	5.1 	&	3.5E-02	&	39.8 	&	3.5E-02	&	120.4 	&	3.4E-02	\\
\cline{2-8}
\multirow{3}{*}{CoR-UTV \cite{kl18}}	&	$\tau=0$	&	\underline{1.4} 	&	1.1E-01	&	\underline{11.5} 	&	1.0E-01	&	\underline{36.2} 	&	1.0E-01	\\
&	$\tau=1$	&	2.3 	&	3.9E-02	&	17.2 	&	3.9E-02	&	55.6 	&	3.8E-02	\\
&	$\tau=2$	&	3.2 	&	3.5E-02	&	23.3 	&	3.5E-02	&	74.6 	&	3.4E-02	\\
\cline{2-8}
\multirow{3}{*}{RP-TSOD \cite{kc20}}	&	$\tau=0$	&	1.9 	&	1.1E-01	&	16.2 	&	1.0E-01	&	54.9 	&	1.0E-01	\\
&	$\tau=1$	&	2.7 	&	3.9E-02	&	22.4 	&	3.9E-02	&	74.5 	&	3.8E-02	\\
&	$\tau=2$	&	3.6 	&	3.5E-02	&	28.6 	&	3.5E-02	&	93.6 	&	3.4E-02	\\
\cline{2-8}
\multirow{3}{*}{PbP-QLP \cite{kc21}}	&	$\tau=0$	&	\textbf{1.3} 	&	1.1E-01	&	13.2 	&	1.0E-01	&	45.7 	&	1.0E-01	\\
&	$\tau=1$	&	2.3 	&	3.9E-02	&	19.2 	&	3.9E-02	&	65.0 	&	3.8E-02	\\
&	$\tau=2$	&	3.0 	&	3.5E-02	&	25.5 	&	3.5E-02	&	87.8 	&	3.4E-02	\\
\cline{2-8}
\multirow{3}{*}{randUTV \cite{mqh2019}}	&	$\tau=0$	&	7.1 	&	3.7E-15	&	52.3 	&	4.0E-15	&	207.8 	&	4.1E-15	\\
&	$\tau=1$	&	7.6 	&	3.7E-15	&	56.7 	&	3.9E-15	&	211.4 	&	4.0E-15	\\
&	$\tau=2$	&	8.7 	&	3.7E-15	&	60.5 	&	3.9E-15	&	216.6 	&	4.0E-15	\\
\cline{2-8}
\multirow{3}{*}{farPCA \cite{fy2023}}	    &	$\tau=0$	&	5.0 	&	3.7E-12	&	41.1 	&	6.2E-11	&	90.0 	&	3.6E-10	\\
&	$\tau=1$	&	8.1 	&	3.0E-13	&	63.6 	&	4.1E-11	&	138.4 	&	2.2E-11	\\
&	$\tau=2$	&	10.2 	&	5.3E-13	&	75.4 	&	4.7E-12	&	166.8 	&	1.3E-11	\\
\cline{2-8}
\multirow{3}{*}{Adaptive PCA \cite{dyxl2020}}	    &	$\tau=0$	&	2.8 	&	1.0E-14	&	22.6 	&	6.0E-14	&	74.1 	&	1.2E-12	\\
&	$\tau=1$	&	3.8 	&	1.3E-14	&	30.4 	&	8.9E-14	&	98.0 	&	2.5E-13	\\
&	$\tau=2$	&	4.5 	&	6.2E-14	&	36.4 	&	6.3E-14	&	117.5 	&	9.4E-13	\\
\cline{2-8}
\multirow{3}{*}{Adaptive RSVD \cite{hmt11}}	    &	$\tau=0$	&	5.5 	&	1.9E-13	&	57.0 	&	9.9E-13	&	188.8 	&	1.5E-13	\\
	                            &	$\tau=1$	&	6.5 	&	4.0E-15	&	64.5 	&	4.9E-15	&	219.8 	&	8.0E-15	\\
	                            &	$\tau=2$	&	7.4 	&	4.0E-15	&	72.0 	&	4.9E-15	&	235.7 	&	8.0E-15	\\
\cline{2-8}
\multirow{3}{*}{FSVD \cite{nsr}}	    &	$\tau=0$	&	1.8 	&	2.3E-12	&	14.1 	&	6.5E-11	&	46.1 	&	8.5E-11	\\
	                            &	$\tau=1$	&	2.7 	&	2.8E-15	&	21.6 	&	3.3E-15	&	68.6 	&	3.7E-15	\\
	                            &	$\tau=2$	&	3.6 	&	2.8E-15	&	28.9 	&	3.3E-15	&	92.3 	&	3.6E-15	\\
\cline{2-8}
\multirow{3}{*}{Algorithm \ref{alg:A}}	&	$\tau=0$	&	1.5 	&	5.8E-13	&	\textbf{10.5} 	&	3.2E-12	&	\textbf{29.3} 	&	4.9E-12	\\
	                                        &	$\tau=1$	&	2.4 	&	1.3E-15	&	18.0 	&	1.3E-15	&	52.3 	&	1.3E-15	\\
&	$\tau=2$	&	3.3 	&	1.2E-15	&	25.6 	&	1.3E-15	&	75.5 	&	1.3E-15	\\
    \hline
  \end{tabular}
\end{table}

From  Tables \ref{0.4n}-\ref{0.35n}, we have  the following observations.
\begin{enumerate}
  \item When the sampling size parameter is greater than or equal to matrix rank, all algorithms achieve high computational accuracy. However, when the sampling size parameter is less than  matrix rank, the computation accuracy of RSVD, TSR-SVD, CoR-UTV, RP-TSOD, and PbP-QLP is significantly reduced, whereas Algorithm \ref{alg:A} maintains high accuracy. This is because in Algorithm \ref{alg:A}, we use Algorithm \ref{alg:qr} that can independently determine the position of the rank and stop computation at that point, avoiding undersampling.
  \item Algorithm \ref{alg:A} demonstrates clear computational advantages over most competing methods in terms of runtime efficiency. This is evident across nearly all tested matrix sizes and sampling parameters, as shown in Tables \ref{0.4n}-\ref{0.35n}.
\end{enumerate}

Next, we investigate the changes in algorithms' computation time and accuracy as matrix rank varies.
We set $n = 8000$ and increase the rank from 400 to 8000.
At the same time, the sampling parameter $d$ is set equal to the rank $r$ to calculate the matrix approximation.
Experiments are conducted using subspace iteration counts of $\tau = 0$, $\tau = 1$, and $\tau = 2$. The comparison results of computation time and accuracy are presented in Figures \ref{8ct-1}-\ref{8ct-3}.

\begin{figure}[!ht]
  \centering
      \includegraphics[width=5.0in]{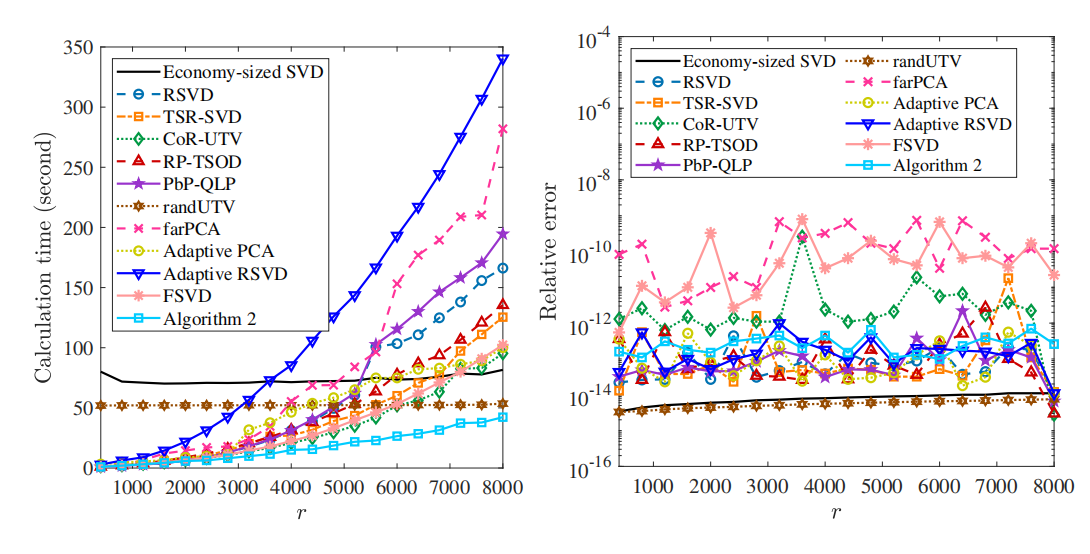}
  \caption{Calculation time and relative error of different algorithms for matrix \eqref{mat2} as the rank $r$ increases, with $n=8000$ and $\tau=0$.}\label{8ct-1}
\end{figure}

\begin{figure}[!ht]
  \centering
      \includegraphics[width=5.0in]{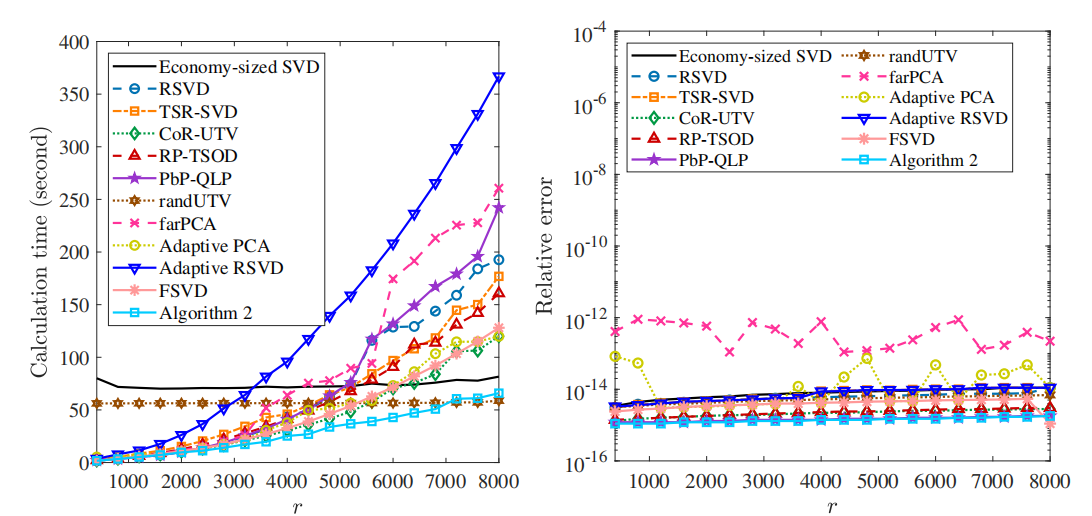}
  \caption{Calculation time and relative error of different algorithms for matrix \eqref{mat2} as the rank $r$ increases, with $n=8000$ and $\tau=1$.}\label{8ct-2}
\end{figure}

\begin{figure}[!ht]
  \centering
      \includegraphics[width=5.0in]{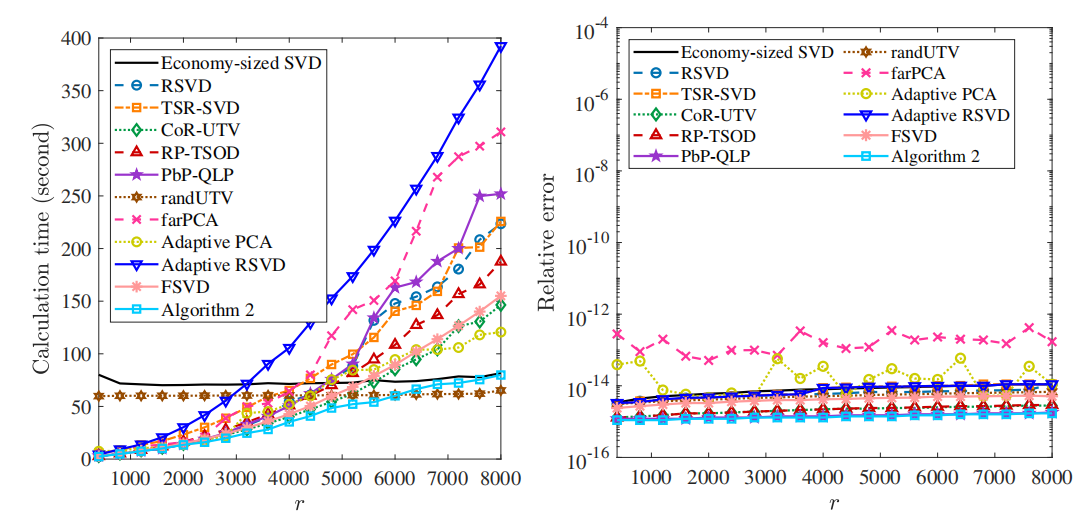}
  \caption{Calculation time and relative error of different algorithms for matrix \eqref{mat2} as the rank $r$ increases, with $n=8000$ and $\tau=2$.}\label{8ct-3}
\end{figure}

From Figures \ref{8ct-1}-\ref{8ct-3}, we  can conclude the following results.

\begin{enumerate}
  \item Algorithm \ref{alg:A} takes less computation time than other compared algorithms. As the matrix rank $r$ increases, the runtime gap between Algorithm \ref{alg:A} and other algorithms further enlarges. Therefore, the advantage of Algorithm \ref{alg:A} in computation time is more obvious when the rank is unknown.
  \item When computing strict low-rank approximations of matrices, relative error values are all small by all mentioned algorithms. 
  \item As shown in Figures \ref{8ct-1}-\ref{8ct-3}, without subspace iteration, the approximation error of Algorithm \ref{alg:A} plateaus at around $10^{-13}$. When subspace iteration is applied, the accuracy improves significantly, reaching approximately $10^{-15}$, which is the best among all compared methods. This improvement is attributed to the large matrix size and the slow decay of singular values. In such cases, subspace iteration enhances the alignment of the sampling space with the dominant singular directions, thereby improving the overall approximation quality.
\end{enumerate}

Therefore, through the above numerical analysis, Algorithm \ref{alg:A} has better performance than other numerical methods when the rank is unknown.

\section{Conclusion}
We propose a fast low-rank matrix approximation algorithm for matrix with unknown rank. The introduced randomized algorithm for basis extraction allows the acquisition of the matrix rank, addressing the limitations of empirically obtaining numerical ranks. The presented algorithms undergo theoretical analysis, and numerical experiments on random matrices indicate that the proposed algorithm exhibits superior time efficiency while maintaining accuracy comparable to existing algorithms.  Power method techniques are incorporated into the algorithms to enhance precision, and experimental results validate the effectiveness of this enhancement.
Finally, the proposed algorithm is applied to image compression and reconstruction, yielding favorable results and outperforming the fastest existing algorithm.


\end{document}